\documentclass[a4paper]{article}
\usepackage{amsmath}
\usepackage{amssymb}
\usepackage{amsthm}
\usepackage{array}
\usepackage{mathrsfs} % to use \mathscr 
\usepackage{mathtools}
\usepackage{hori}
\usepackage{hyperref}
\usepackage{color}
\usepackage{%
    algorithm,
    algorithmic
}
\usepackage{subcaption}

\renewcommand{\Re}{I\!\!R} % Real space (IR form)
\newcommand{\NEP}{$\mathrm{NEP}(X,\{\Theta^\nu\}_{\nu=1}^N)$}
\newcommand{\NEPeps}{$\mathrm{NEP}(X,\{\Theta^\nu_\varepsilon\}_{\nu=1}^N)$}
\newcommand{\NEPepsk}{$\mathrm{NEP}(X,\{\Theta^\nu_{\varepsilon_k}\}_{\nu=1}^N)$}

\newcommand{\diag}{\mathrm{diag}}
\newcommand{\paper}[6]{{#1}, {#3}, \emph{#4} {\bf #5}, pp.{#6}, {#2}.} % (authors, year, title, journal, volume, pages)
\newcommand{\book}[4]{{#1}, {\emph{#3}}, {#4}, {#2}.} % (authors, year, title, publisher)
\newcommand{\chp}[7]{{#1}, {#3}, In: {#4} (eds) {#5}, {#6}, pp.{#7}, {#2}.} % (authors, year, title, editors, booktitle, publisher, pp.)
\newcommand{\conf}[6]{{#1}: {#3}. In: \emph{#4}, {#5}, pp.{#6} ({#2})} % (authors, year, title, conference name, location, page)

\definecolor{myred}{RGB}{255,50,50}         % revision version
\definecolor{myblack}{RGB}{0,0,0}           % normal version

\title{A method for multi-leader--multi-follower games by smoothing the followers' response function}
\author{%
    Atsushi Hori\thanks{Faculty of Science and Technology, Seikei University, Tokyo, Japan.} \and
    Daisuke Tsuyuguchi\thanks{Wakayama Prefectural Board of Education, Wakayama, Japan.} \and
    Ellen H. Fukuda\thanks{Graduate School of Informatics, Kyoto University, Kyoto, Japan.}
}

\date{\today}

\begin{document}

\maketitle

\begin{abstract}
The multi-leader--multi-follower game (MLMFG) involves two or more leaders and followers and serves as a generalization of the Stackelberg game and the single-leader--multi-follower game (SLMFG).
Although MLMFG covers wide range of real-world applications,
%The SLMFG may be reformulated into a mathematical program with equilibrium constraints (MPEC) which has been extensively studied.
its research is still sparse.
Notably, fundamental solution methods for this class of problems remain insufficiently established.
A prevailing approach is to recast the MLMFG as an equilibrium problem with equilibrium constraints (EPEC) and solve it using a solver.
Meanwhile, interpreting the solution to the EPEC in the context of MLMFG may be complex due to shared decision variables among all leaders, followers' strategies that each leader can unilaterally change, but the variables are essentially controlled by followers.
To address this issue, we introduce a response function of followers' noncooperative game that is a function with leaders' strategies as a variable.
Employing this approach allows the MLMFG to be solved as a single-level differentiable variational inequality using a smoothing scheme for the followers' response function.
We also demonstrate that the sequence of solutions to the smoothed variational inequality converges to a stationary equilibrium of the MLMFG.
Finally, we illustrate the behavior of the smoothing method by numerical experiments and confirm its validity.
\end{abstract}

%\subclass{91A65 \and 91A10 \and 90C33}

%All acknowledgements should be placed in the back of the paper after Conclusions..
%\linenumbers
\section{Introduction}

% Def. of MLF & applications
The multi-leader--multi-follower game (MLMFG) is a bilevel structured noncooperative game featuring two or more leaders who determine their strategies first, followed by two or more followers who make their choices in a strategic setting.
This framework can be viewed as an extension of bilevel optimization, the Stackelberg model, and single-leader--follower games.
The inherent complexity of this class arises from the fact that each leader's (upper-level) optimization problem is constrained by a set of followers' (lower-level) Nash equilibria, which is computationally challenging to evaluate.

With the growing focus on noncooperative game theory, the MLMFG is also studied in economics and computer science.
The MLMFG has often been used to analyze deregulated markets, such as wholesale electricity markets, which consist of several energy companies (leaders) and the independent system operator (follower) \cite{Hobbs2000,Leyffer2010,Pang2005}.
The class is formulated as a multi-leader--single-follower game (MLSFG).
More recently, researchers in computation and telecommunication formulated an edge computing model with the MLMFG to achieve the best computation resource allocation \cite{Chen2020,Jiang2020,Xiong2019}.
In edge computing, leaders serve as edge computers with medium-scale computational resources, and followers play terminals such as smartphones, security cameras, or robot arms in a factory.
For more recent advances and applications of the MLMFG, please refer to the surveys by Aussel and Svensson \cite{Aussel2020}, and Hu and Fukushima \cite{Hu2015}.

% Previous studies
Theoretical studies on MLMFG have taken two main directions.
The first approach reformulates the followers' problems into necessary conditions for optimality, known as the Karush--Kuhn--Tucker (KKT) conditions and incorporates it into the constraint of each leader's optimization problem; that is, each leader's problem solves a mathematical program with equilibrium constraints (MPEC) \cite{Luo1996}.
The resultant problem is referred to as an equilibrium problem with equilibrium constraints (EPEC) \cite{Outrata1998,Su2005} and is solved with an MPEC solver.
This approach in terms of MLMFG has also extensively studied over the years \cite{Hori2019,Leyffer2010,Su2005}.
The EPEC approach yields the so-called shared constraints and variables, which coincides the KKT conditions and all followers' strategies, respectively.
However, the EPEC formulation introduces shared variables that can complicate the interpretation within the MLMFG context because the variables may be unilaterally changed by each leader as they desires, though the variables are essentially followers' strategies.

The second approach addresses this issue by considering the best response of the followers' noncooperative game given by leaders' strategies and integrating them into each leader's optimization problem \cite{Herty2022,Hu2011,Hu2013}; we call this technique the best response approach.
The problem does not explicitly make the followers' strategies to appear; that is, the resultant problem is a simple Nash equilibrium problem among leaders. In general the resultant problem still has complicated objective functions, non-smooth and non-convex in each leader's optimization problem, but this approach allows us to adopt a usual technique used for solving Nash equilibrium problems by smoothing the followers' response.

The smoothing method for response functions have particularly been a focus in MPEC over the decades.
To the best of our knowledge, Facchinei et al. \cite{Facchinei1999} first considered the smoothing method for the response function in the MPEC, which means that the response function is characterized by the lower-level equilibrium constraint parametrized by the upper-level variable.
They showed the convergence to the Clarke stationary point of the MPEC as the smoothing parameter converges to 0.
Chen and Fukushima \cite{Chen2004} then proposed the same method for the MPEC where the lower-level equilibrium constraint is a P-matrix linear complementarity constraint, and they showed the convergence to the Bouligand stationary point of the MPEC.
Hu and Fukushima \cite{Hu2012} extended this approach to EPEC, confirming convergence to the Bouligand stationary point, which satisfies the Bouligand stationarity for each MPEC.
To date, however, no studies have confirmed such convergence for MPEC and EPEC with nonlinear complementarity constraints at the lower-level.

To the best our knowledge, the best response approach in the class of the MLMFG has only been studied in the case of quadratic games where the followers' response is written in closed form by Hu and Fukushima \cite{Hu2011,Hu2013} and Herty et al. \cite{Herty2022}.
They demonstrated the existence of the leader--follower Nash equilibrium, where no one has incentive to change their strategy in both levels, for the quadratic game; each player solves a convex quadratic programming problem.
Hu and Fukushima \cite{Hu2011} considered the quadratic game when one follower solves equality constrained convex quadratic programming.
In \cite{Hu2013}, they then considered the same class of the MLMFG under uncertainty and demonstrated the existence and uniqueness of robust Nash equilibrium.
Herty et al. \cite{Herty2022} extended the class to which one follower solves linear inequality constrained convex quadratic programming, and they proposed the smoothing technique of the follower's response function.

However, there are many cases that cannot be formulated as a quadratic game in real-world applications, and in such a case, the followers' response may no longer be obtained explicitly. 
For example, as often used in micro economics, the utility function is often characterized by a logarithm or an exponential function.
Not only in economics but in optimal resource allocation in edge computing, Lyu et al. \cite{Lyu2022} used log-utility function for follower's optimization.
Moreover, in the blockchain based cloud/edge computing network, the utility function for followers is characterized by a fraction, e.g., Xiong et al. \cite{Xiong2019}.
%Hori and Fukushima \cite{Hori2019} established a Gauss--Seidel-based diagonalization with penalization scheme to obtain a stationary equilibrium of the multi-leader--follower game.
%Although they also proposed a refinement procedure to avoid an ill-condition, they do not provide that the resultant point admits a stationary equilibrium.
%To achieve a balance between theoretical and numerical analysis, we propose an smoothing method for the multi-leader--follower game via variational inequality reformulation.
%Herty et al. \cite{Herty2022} proposed a smoothing method for a multi-leader--single-follower game in which each leader and follower solve quadratic optimization problems and the response function of the follower is explicitly written. 

% Problem issues
Currently, the study on the existence of the equilibria in the MLMFG is very limited with the best response approach because demonstrating the existence of leader--follower Nash equilibrium essentially requires that the each leader's problem is convex.
However, identifying the convexity of the problem often requires the response function to be explicitly written, which can be very difficult except in special cases as we introduced above 
\cite{Herty2022,Hu2011,Hu2013}.

In this paper, we propose the best response approach for a more general class of MLMFG with a smoothing method based on Facchinei et al. \cite{Facchinei1999}.
Using this technique allows MLMFG to be solved by a single-level (differentiable) variational inequality regarding the leaders' Nash game.
Meanwhile, in this class there may not exist the (global) leader--follower Nash equilibrium, and also if it exists, finding the equilibrium is NP-hard in general since the objective function of each leader is not convex.
Hence, we concentrate on a weaker concept of the equilibrium, \emph{stationary Nash equilibrium}, as the first-order condition for the \emph{local} leader--follower Nash equilibrium.
Then we demonstrate that the solution of the (smoothed) approximated variational inequality converges to the stationary Nash equilibrium of MLMFG as the smoothing parameter gets $0$.
We also report the results of numerical experiments conducted and illustrate the behavior of the proposed method.

In summary, our contributions bridge the following gaps in existing studies:
\begin{enumerate}
    \item Unlike \cite{Herty2022} and \cite{Hu2011}, which focus on cases with explicitly calculated follower(s)' response, we extend their framework to cases where the explicit form of follower(s)' response is not analytically obtainable; %which intrinsically implies that one has no need to solve the follower(s)' optimization problem numerically.
    \item In contrast to \cite{Hu2011}, which considers only equality constraints in the followers' optimization problems, our paper considers cases involving nonlinear inequality constraints;
    \item This paper may be regarded as the first to demonstrate convergence to the Bouligand stationary point in a response-based smoothing method for EPECs with nonlinear complementarity constraints associated with MLMFG.
\end{enumerate}

The remainder of this paper is organized as follows: Section~\ref{sec:prelim} outlines mathematical concepts and basic noncooperative theory used in our approach.
Section~\ref{sec:MLF} describes the MLMFG and its solution concepts.
Section~\ref{sec:smoothing} introduces a smoothing method for the MLMFG and then analyzes its convergence to stationary Nash equilibrium.
Section~\ref{sec:exp} reports on numerical experiments that illustrate the effectiveness of our proposed method with a toy example.
Finally, Section~\ref{sec:conclusion} offers some concluding remarks.

\section{Preliminaries}\label{sec:prelim}

%\red{Ellen: There are some notations in the next section, so maybe put everything in Section~2?\\}
This section provides some fundamental concepts about convex analysis and Nash games.
Throughout this paper we use the following notations:
%For simplicity of notation, let $(x,y)\in\Re^{n+m}$ denote the column vector $(x^\top,y^\top)^\top$, which is the concatenation of vectors $x\in\Re^n$ and $y\in\Re^m$.
Let $F\colon\Re^n\to\Re^m$ be differentiable, $\nabla F(x):=[\nabla F_1(x),\dots,\nabla F_m(x)]$ is the transposed Jacobian matrix of $F$ at $x \in \Re^n$; we simply call $\nabla F(x)$ the Jacobian matrix of $F(x)$.
%For a nonsingular matrix $M$, $M^{-\top}:=(M^{-1})^\top=(M^\top)^{-1}$.
For vectors $a\in\Re^n$ and $b\in\Re^n$, $a\perp b$ denotes $a^\top b=0$.

\begin{definition}[{\cite[Definition 2.6.1]{Clarke1990}}]\label{def:generalized.Jacobian}
    The (transposed) \emph{generalized Jacobian} of $F\colon\Re^n\to\Re^m$ at $x$, denoted as $\partial F(x)$, is the convex hull of all $n\times m$ matrices $W$ obtained as the limit of a sequence of the form $\nabla F(x^k)$, where $x^k\to x$ and $x^k\in\mathcal{D}_F$.
    Here, $\mathcal{D}_F\subset\Re^n$ is the set of which $F$ is differentiable.
\end{definition}
Symbolically, one has
$$
    \partial F(x)=\co\left\{\lim_{x^k\to x}\nabla F(x^k)\ \middle|\ x^k\in\mathcal{D}_F\right\},
$$
where $\co$ denotes the convex hull of a set.

For a real-valued function $\psi\colon\Re^n\to\Re$, a \emph{directional derivative} $\psi'(x;d)$ of $\psi$ at $x\in\Re^n$ in the direction $d\in\Re^n$ is defined to be
\begin{align}\label{eq:directional.derivative}
    \psi'(x;d):=\lim_{\tau\downarrow0}\frac{f(x+\tau d)-f(x)}{\tau},
\end{align}
when the limit exists.
The \emph{Clarke generalized directional derivative} $\psi^\circ(x;d)$ of the function $\psi$ at $x$ in the direction $d\in\Re^n$ is defined as
$$
    \psi^\circ(x;d):=\underset{y\to x,\tau\downarrow0}{\lim\sup}\frac{\psi(y+\tau d)-\psi(y)}{\tau}.
$$

The regularity of a real-valued function is defined as follows.
\begin{definition}\label{def:regular}
    A function $\psi\colon\Re^n\to\Re$ is \emph{regular} at $x\in\Re^n$ if, for every $d\in\Re^n$, the directional derivative $\psi'(x;d)$ exists and satisfies
    $$
        \psi'(x;d)=\psi^\circ(x;d).
    $$
    Moreover, a vector-valued function $\Psi\colon\Re^n\to\Re^m$ is regular if each element of the function $\psi_i$, $i=1,\dots,m$, is regular.
\end{definition}

The tangent cone ${\cal T}_X(x)$ of $X\subset\Re^n$ at $x$ is defined by\small
$$
    {\cal T}_X(x):=\left\{d\in\Re^n\ \middle|\ d=\lim_{\nu\to\infty}\alpha_\nu(x^\nu-x),\lim_{\nu\to\infty}x^\nu=x,x^\nu\in X,\alpha_\nu\ge0,\nu=1,2,\dots \right\},
$$\normalsize
and the normal cone ${\cal N}_X(x)$ of $X\subset\Re^n$ at $x$ is defined as the set of points $v$ if there exist sequences $\{x^k\}\subset X$ and $\{v^k\}$ with
$$
    x^k\to x,\ v^k\to v,\ v^k\in{\cal T}_X(x^k)^\circ\ \ \forall k,
$$
where ${\cal T}_X(x)^\circ:=\{y\in\Re^n\mid\langle y,z\rangle\le0\ \ \forall z\in{\cal T}_X(x)\}$ denotes the polar cone of ${\cal T}_X(x)$.

%The regularity of a set over $\Re^n$ is defined as follows.
\begin{definition}\label{def:regular.set}
    The set $X\subset\Re^n$ is \emph{regular at $x\in X$} if
    $$
        {\cal N}_X(x)={\cal T}_{X}(x)^\circ.
    $$
    Furthermore, $X$ is (simply called) \emph{regular} if $X$ is regular at all $x\in X$.
\end{definition}

The following lemma is a sufficient condition for the regularity of a set.
\begin{lemma}[{\cite[Proposition 4.6.3]{Bertsekas2003}}]\label{lem:regular.sufficient}
    If $X\subset\Re^n$ is convex, then $X$ is regular, and the normal cone, ${\cal N}_X(x)={\cal T}_X(x)^\circ$ under regularity, of $X$ at $x\in X$ is equivalent to
    $$
        {\cal N}_X(x)=\{d\in\Re^n\mid \langle d,z-x\rangle\le0\quad\forall z\in X\}.
    $$
\end{lemma}

Next we consider an $N$-player Nash equilibrium problem (NEP).
Player labeled with $\nu\in\{1,\dots,N\}$ has $x^\nu\in\Re^{n_\nu}$ as the strategy vector and $X^\nu\subset\Re^{n_\nu}$ as the strategy set.
Player $\nu$ solves the following optimization problem:
\begin{align}\label{prob:i-th}
    \min_{x^\nu\in\Re^{n_\nu}}\theta^\nu(x^\nu,x^{-\nu}) \qquad \text{s.t. } x^\nu\in X^\nu,
\end{align}
where $x^{-\nu}:=(x^1,\dots,x^{\nu-1},x^{\nu+1},\dots,x^N)\in\Re^{n-n_\nu}$ denotes a tuple of strategies except player $\nu$'s one.
Let $n:=n_1+\dots+n_N$ and the function $\theta^\nu\colon\Re^n\to\Re$ be continuously differentiable.
%Each player $\nu$ chooses a strategy in $X^\nu$ that minimizes their own cost function.

\begin{definition}[{Nash equilibrium}]
    A tuple of strategies $x^*:=(x^{*,1},\dots,x^{*,N})$ is called a \emph{Nash equilibrium} if for each $\nu$,
    $$
        x^{*,\nu}\in\arg\min_{x^\nu\in X^\nu}\theta^\nu(x^\nu,x^{*,-\nu}).
    $$
\end{definition}

In other words, a Nash equilibrium is a tuple of strategies in which no one can reduce their cost unilaterally.
However, a Nash equilibrium does not always exist in general, and it is difficult to find even if it exists.

The NEP may be characterized by a variational inequality (VI).
Let $\theta^\nu$ be differentiable, and define $X:=X^1\times\dots\times X^N$ and
\begin{align}\label{vi:Fmapping}
    F(x):=\left[\begin{array}{c}\nabla_{x^1}\theta^1(x^1,x^{-1})\\ \vdots \\ \nabla_{x^N}\theta^N(x^N,x^{-N}) \end{array}\right].
\end{align}

\begin{proposition}[{Proposition 1.4.2 \cite{Facchinei2003}}]\label{prop:Nasheq.VI}
    Assume that $\theta^\nu(\cdot,x^{-\nu})$ is convex for any $x^{-\nu}\in\Re^{n-n_\nu}$, and $X^\nu\subset\Re^{n_\nu}$ is nonempty, closed, and convex.
    Then, a tuple of strategies $x^*$ is a Nash equilibrium if and only if $x^*$ is a solution to the following VI:
    \begin{align}\label{vi.ne}
        \langle F(x^*), x-x^*\rangle\ge0\quad\forall x\in X.
    \end{align}
\end{proposition}

%Regardless of the convexity on problem \eqref{prob:i-th}, the solution to VI \eqref{vi.ne} is referred to as a \emph{stationary Nash equilibrium} \cite{Hu2011}.
Notice that even if each player's optimization \eqref{prob:i-th} is convex, the existence of Nash equilibrium is not guaranteed.
In other words, the solution to variational inequality \eqref{vi.ne} may not exist in general.
The following proposition ensures the existence of Nash equilibrium in an $N$-player NEP.

\begin{proposition}[{Aubin \cite{Aubin1979}}]\label{prop:existence.Nash.eq}
    Suppose that the assumptions of Proposition~\ref{prop:Nasheq.VI} hold.
    Assume that $X^\nu\subset\Re^{n_\nu}$ is compact for all $\nu\in\{1,\dots,N\}$.
    Then, the Nash equilibrium of the NEP in which player $\nu$ solves \eqref{prob:i-th} exists.
\end{proposition}

The uniqueness of Nash equilibrium is stated as below.

\begin{proposition}\label{prop:uniquenss}
    Suppose that the assumptions of Proposition~\ref{prop:existence.Nash.eq} hold.
    Assume that the mapping $F\colon\Re^n\to\Re$ defined in \eqref{vi:Fmapping} is strictly monotone on $X\subset\Re^n$, i.e., 
    $$
        \langle F(x)-F(x'),x-x'\rangle>0\quad\forall x,x'\in X \text{ such that } x\neq x'.
    $$
    Then, the solution to VI \eqref{vi.ne} is unique, and it is a Nash equilibrium.
\end{proposition}
\begin{proof}
    \cite[Theorem 2.3.3]{Facchinei2003} ensures that the strictly monotone VI has at most one solution.
    By the existence result from Proposition~\ref{prop:existence.Nash.eq}, the Nash equilibrium uniquely exists.
\end{proof}

\section{The multi-leader--multi-follower games}\label{sec:MLF}

% Problem definition
Consider an MLMFG of $N$ leaders and $M$ followers.
Let $X^\nu\subset\Re^{n_\nu}$ and $\theta^\nu\colon\Re^{n+m}\to\Re$ be the strategy set and cost function of leader $\nu\in\{1,\dots,N\}$, respectively, where $m:=m_1+\dots+m_M$. 
Let $Y^\omega(x)\subset\Re^{m_\omega}$ and $\gamma^\omega\colon\Re^{n+m}\to\Re$ be the strategy set and cost function of follower $\omega\in\{1,\dots,M\}$, respectively.

For a fixed all followers' strategies $y\in\Re^m$, determined in the future, leader $\nu$ solves the following optimization problem: 
\begin{align}\label{prob:leader.i-th}
    \min_{x^\nu\in\Re^{n_\nu}}\theta^\nu(x^\nu,x^{-\nu},y)\qquad\text{s.t. }x^\nu\in X^\nu.
\end{align}
After all leaders simultaneously determine their strategies $x\in X:=X^1\times\dots\times X^N$, follower $\omega$ solves the following optimization problem:
\begin{align}\label{prob:follower.j-th}
    \min_{y^\omega\in\Re^{m_\omega}}\gamma^\omega(x,y^\omega,y^{-\omega})\qquad\text{s.t. } y^\omega\in Y^\omega(x).
\end{align}
We can also consider the case in which the constraint $Y^\omega(x)$ of follower $\omega$'s problem also depends on $y^{-\omega}$, i.e., $Y^\omega(x,y^{-\omega})$, referred to as a \emph{generalized} Nash equilibrium problem (GNEP).
Finding an equilibrium of GNEP, however, is also technical even in a single-level Nash game, which is not the scope of this paper.
Let ${\cal S}(x)$ be a set of Nash equilibria in followers' Nash game. The equilibrium concept of the MLMFG is considered as follows \cite{Hu2015}.

% Optimistic/Pessimistic leader--follower eq.
%Let ${\cal S}(x)$ be the set of Nash equilibria for 

\begin{definition}\label{def:LFNash}
    A tuple of strategies $(x^*,y^*)=(x^{*,1},\dots,x^{*,N},y^{*,1},\dots,y^{*,M})\in X\times{\cal S}(x^*)$ is referred to as a \emph{pessimistic leader--follower (LF) Nash equilibrium} if the following conditions simultaneously hold:
    \begin{align}\label{def.eq:pessimistic.LF.Nash}
        x^{*,\nu}\in\underset{x^\nu\in X^\nu}{\arg\min}\max_{y\in{\cal S}(x^\nu,x^{*,-\nu})}\theta^\nu(x^\nu,x^{*,-\nu},y)\quad\forall\nu\in\{1,\dots,N\},
    \end{align}
    
    A tuple of strategies $(x^*,y^*)=(x^{*,1},\dots,x^{*,N},y^{*,1},\dots,y^{*,M})\in X\times {\cal S}(x^*)$ is referred to as a \emph{optimistic leader--follower (LF) Nash equilibrium} if the following conditions simultaneously hold:
    \begin{align}\label{def.eq:optimistic.LF.Nash}
        x^{*,\nu}\in\underset{x^\nu\in X^\nu}{\arg\min}\min_{y\in{\cal S}(x^\nu,x^{*,-\nu})}\theta^\nu(x^\nu,x^{*,-\nu},y)\quad\forall\nu\in\{1,\dots,N\},
    \end{align}

    If ${\cal S}(x)$ is a singleton for any $x$, i.e, there exists a unique followers' Nash equilibrium for any given leaders' strategies, both equilibrium concepts are equivalent; hence we simply call the equilibrium point a \emph{leader--follower (LF) Nash equilibrium}.
\end{definition}

Unfortunately, the pessimistic LF Nash equilibrium may not exist even if $\theta^\nu$ is continuous and $X^\nu$ is compact since
$$
\varphi(x^\nu,x^{-\nu})=\max_{y\in{\cal S}(x^\nu,x^{-\nu})}\theta^\nu(x^\nu,x^{-\nu},y)
$$
is not necessarily lower semicontinuous with respect to $x^\nu$, which implies that there may not exist the minimizers of $\varphi(x^\nu,x^{-\nu})$.
In this paper, we impose that ${\cal S}(x)$ is a singleton for every $x\in X$ to avoid such a complicated situation; that is, $y$ is uniquely determined depending on $x$.
Then, to emphasize that $y$ is a function of $x$, we rewrite ${\cal S}(x)$ as $y(x)$ and call it a \emph{response} function.
The sufficient condition for the uniqueness of followers' Nash equilibrium will be given later.

%In other words, $(x^*,y^*)$ is an LF Nash equilibrium if $x^*$ is a GNE of the leaders' GNEP for fixed $y^*$, and $y^*$ is a Nash equilibrium of the followers' NEP for fixed $x^*$.

%In this paper, the Nash equilibrium for the lower-level (followers') game uniquely exists depending on the tuple of leaders' strategies $x\in X$.

Plugging $y(x)$ into each leader's problem \eqref{prob:leader.i-th} leads that the MLMFG comprised of \eqref{prob:leader.i-th}--\eqref{prob:follower.j-th} can be reformulated to the following single-level Nash equilibrium problem among leaders: Leader $\nu\in\{1,\dots,N\}$ solves
\begin{align}\label{prob:leader.i-th.reduced}
    \min_{x^\nu\in\Re^{n_\nu}}\quad\Theta^\nu(x^\nu,x^{-\nu}):=\theta^\nu(x^\nu,x^{-\nu},y(x^\nu,x^{-\nu}))\qquad\text{s.t.}\quad x^\nu\in X^\nu.
\end{align}
We call \eqref{prob:leader.i-th.reduced} a \emph{reduced} problem of \eqref{prob:leader.i-th}, and the single-level game in which leader $\nu$ solves \eqref{prob:leader.i-th.reduced} is defined as {\NEP}.
By the definition of response function $y(x)$, the following statement immediately holds.
\begin{proposition}\label{prop:equivalence.GNEP.MLF}
    Let $x^*\in X$ be a Nash equilibrium of {\NEP}.
    Then, $(x^*,y(x^*))$ is an LF Nash equilibrium of the MLMFG.
\end{proposition}
%\begin{proof}
%    By the definition of $y(x)$, $y(x^*)$ is a Nash equilibrium of the followers' NEP for the tuple of leaders' strategies $x^*$.
%    If $x^*$ is a GNE of $\mathrm{GNEP}_0(X^\nu,\theta^\nu)_{i=1}^N$, $x^{*,\nu}$ satisfies \eqref{def.eq:LF.Nash.leader} for every $\nu$.
%    Then, Definition~\ref{def:LFNash} holds.
%    We have completed the proof.
%\end{proof}

Proposition~\ref{prop:equivalence.GNEP.MLF} indicates that under the uniqueness of followers' Nash equilibrium $y(x)$, it is enough to only consider {\NEP} instead of the MLMFG \eqref{prob:leader.i-th}-\eqref{prob:follower.j-th}.
By utilizing the reduction technique into a single-level NEP, the existence of LF Nash equilibrium can be stated as below.
\begin{proposition}\label{prop:existence.L/F.Nash}
    Assume the following conditions:
    \begin{itemize}
        \item For any tuple of leaders' strategies $x\in X$, there exists a unique lower-level response $y(x)$;
        \item For any $\nu$, the leaders' objectives $\theta^\nu$ and the best response function $y$ are continuous;
        \item For any $\nu$, the strategy set $X^\nu$ is nonempty, convex and compact;
        \item For any $\nu$, the composition function $\Theta^\nu(x^\nu,x^{-\nu})=\theta^\nu(x^\nu,x^{-\nu},y(x^\nu,x^{-\nu}))$ is convex with respect to $x^\nu$ for any fixed $x^{-\nu}$.
    \end{itemize}
    Then, an LF Nash equilibrium of the MLMFG exists.
\end{proposition}
\begin{proof}
    The assertion is immediately shown by Proposition~\ref{prop:existence.Nash.eq}.
\end{proof}

% Existence of leader--follower equilibrium in MLFG
%In fact, as Aussel and Svensson \cite{Aussel2020} has already pointed out, the existence results for MLMFG are scarce.
%Besides, most of them are based on a technique to reduce the MLMFG to a Nash equilibrium problem by incorporating the unique lower-level response into the leaders' objectives, and then trying to prove some good properties of the resultant (single-level) Nash equilibrium problem.
%That is, by showing the existence of the Nash equilibrium in \NEP, the existence of LF Nash equilibrium is stated from the observation of Proposition~\ref{prop:equivalence.GNEP.MLF}.

Regrettably, verifying the convexity of $\Theta^\nu$ is intrinsically hard since the lower-level response $y(x)$ may not be written explicitly in general;
in some special cases, however, it is possible, e.g., see Hu and Fukushima \cite{Hu2013}, Sherali \cite{Sherali1984}, and Herty et al. \cite{Herty2022}.

These facts lead that the existence of Nash equilibrium is not guaranteed, and if it exists, finding it is NP-hard in general.
Hence, we concentrate on a weaker concept of Nash equilibrium as stated below.
The following concept is derived from a Bouligand stationarity for a mathematical program with equilibrium constraints (MPEC) \cite[Lemma 4.2.5]{Luo1996} and an equilibrium problem with equilibrium constraints (EPEC) \cite{Hu2012}; we extended the concept to \NEP.
\begin{definition}\label{def:stationary.Nash}
    A tuple of leaders' strategies $x^*\in X$ is called a \emph{Bouligand (B-) stationary Nash equilibrium} of {\NEP} if, for every $\nu\in\{1,\dots,N\}$, $x^{*,\nu}\in\Re^{n_\nu}$ satisfies
    \begin{align}\label{ieq:def.B-stationary}
    \begin{aligned}
        &(\Theta^\nu)'(x^{*,\nu},x^{*,-\nu};d^\nu)=\nabla_{x^\nu}\theta^\nu(x^{*,\nu},x^{*,-\nu},y(x^{*,\nu},x^{*,-\nu}))^\top d^\nu\\
        &+y'_{x^\nu}(x^{*,\nu},x^{*,-\nu};d^\nu)^\top\nabla_y\theta^\nu(x^{*,\nu},x^{*,-\nu},y(x^{*,\nu},x^{*,-\nu}))\ge0\quad\forall d^\nu\in{\cal T}_{X^\nu}(x^{*,\nu}),
    \end{aligned}
    \end{align}
    where $y'_{x^\nu}(x^\nu,x^{-\nu};d^\nu)\in\Re^m$ is a partial directional derivative of $y$ with respect to $x^\nu$ along the direction $d^\nu\in\Re^{n_\nu}$ in the sense of \eqref{eq:directional.derivative}.
    %there is a matrix $V^{*,\nu}\in\partial_{x^\nu}y(x^{*,\nu},x^{*,-\nu})\subset\Re^{n_\nu\times m}$ such that
    %$$
    %\begin{aligned}
    %    \nabla_{x^\nu}\theta^\nu&(x^{*,\nu},x^{*,-\nu},y^*)^\top d^\nu+\\
    %    &\nabla_y\theta^\nu(x^{*,\nu},x^{*,-\nu},y^*)^\top V^{*,\nu}d^\nu\ge0\quad\forall d^\nu\in{\cal T}(x^{*,\nu};X^\nu),
    %\end{aligned}
    %$$
    %where $y^*:=y(x^{*,\nu},x^{*,-\nu})$ and ${\cal T}(x^{*,\nu};X^\nu(x^{*,-\nu}))$ denotes the tangent cone of $X^\nu(x^{*,-\nu})$ at $x^{*,\nu}$.
\end{definition}

We also define a weaker concept of stationary Nash equilibrium, which is derived from a Clarke stationarity in nonsmooth analysis \cite{Bertsekas2003}.

\begin{definition}\label{def:weak.stationary.Nash}
    A tuple of leaders' strategies $x^*\in X$ is called a \emph{Clarke (C-) stationary Nash equilibrium} of {\NEP} if, for every $\nu\in\{1,\dots,N\}$, $x^{*,\nu}\in\Re^{n_\nu}$ satisfies:
    $$
    \begin{aligned}
        0\in\partial_{x^\nu}\Theta^\nu(x^{*,\nu},x^{*,-\nu})+{\cal T}_{X^\nu}(x^{*,\nu})^\circ.
    \end{aligned}
    $$
\end{definition}

\section{Smoothing methods and its convergence to stationary Nash equilibrium}\label{sec:smoothing}

%There are mainly two methods to obtain the followers' response $y(x)$: 1. KKT reformulation, 2. Variational inequality reformulation.
%The first approach requires followers' problems to hold a constraint qualification, but obtaining KKT points for fixed leaders' strategies are relatively easy.
%The second approach does not require any constraint qualifications on followers problems, but solving VI is difficult and the smoothing method cannot be applied in this approach; hence, it intrinsically requires the strict complementarity at a solution for VI in practice\footnote{Some ill-posed problems do not satisfy the strict complementarity, but not a few problems satisfy it.}.
%In this paper, we take the first approach.

Since the response function $y(x)$ is nonsmooth, it is difficult to obtain the B-/C-stationary Nash equilibrium of {\NEP} numerically.
To overcome this, we propose a smoothing method and show that as the smoothing parameter decreases, the sequence of stationary Nash equilibria to the smoothed NEP converges to the B-/C-stationary Nash equilibrium of {\NEP}.

\subsection{Smoothing method}

Hereinafter, the strategy set $Y^\omega(x)\subset\Re^{m_\omega}$ of follower $\omega$ is defined by
$$
    Y^\omega(x):=\{y^\omega\in\Re^{m_\omega}\mid g^\omega(x,y^\omega)\le0\},\quad\omega\in\{1,\dots,M\},
$$
where $g^\omega(x,\cdot)\colon\Re^{m_\omega}\to\Re^{l_\omega}$.
Let $l:=l_1+\dots+l_M$.
Note that we omit the equality constraints in the model since it is not essential in the analysis.

In the following, we also assume the conditions stated below.

\begin{assumption}\label{asmp:LF}
For all $\nu\in\{1,\dots,N\}$, the following conditions hold:
\begin{enumerate}[(L1)]
    \item $\theta^\nu$ is continuously differentiable;\label{asmp:leader1}
    \item The set $X^\nu\subset\Re^{n_\nu}$ is nonempty and compact.\label{asmp:leader2}
    %\item $X^\nu\colon\Re^{n-n_\nu}\rightrightarrows\Re^{n_\nu}$ is a continuous set-valued mapping, and for arbitrary $x^{-\nu}\in\Re^{n-n_\nu}$, $X^\nu(x^{-\nu})$ is nonempty convex compact;
    %\item There exists a nonempty, convex, compact set $K^\nu\subset\Re^{n_\nu}$ such that the set-valued map $X^\nu\colon K^{-\nu}\rightrightarrows K^\nu$ is both upper and lower semicontinuous, where $K^{-\nu}:=\Pi_{i'\neq i}K^\nu$.
    %\item There exists a open bounded convex set ${\cal X}$ such that $X^\nu(x^{-\nu})\subset{\cal X}$.
\end{enumerate}
In addition, for all $\omega\in\{1,\dots,M\}$, the following conditions hold:
\begin{enumerate}[(F1)]
    \item $\gamma^\omega$ and $g^\omega$ are sufficiently smooth, and $\gamma^\omega(x,\cdot,y^{-\omega})$ is convex for arbitrary given $x$ and $y^{-\omega}$;\label{asmp:follower1}
    \item $Y^\omega(x)$ is nonempty, convex and compact;\label{asmp:follower2}
    \item For any given $x\in X$ and any feasible solution $y^\omega\in Y^\omega(x)$, the linear independence constraint qualification (LICQ) for the inequality constraint $g^\omega(x,y^\omega)\le0$ holds.\label{asmp:follower3}
\end{enumerate} 
\end{assumption}

%\paragraph{Remark}  If $X^\nu\subset\Re^{n_\nu}$ is convex, the regularity assumption automatically holds.

Let 
$$
    G(x,y):=\left[
    \begin{array}{c}
        \nabla_{y^1}\gamma^1(x,y^1,y^{-1}) \\
        \vdots \\
        \nabla_{y^M}\gamma^M(x,y^M,y^{-M})
    \end{array}
    \right],\quad
    Y(x):=Y^1(x)\times\dots\times Y^M(x).
$$
Under the convexity assumption on each follower's optimization problem \eqref{prob:follower.j-th}, the condition for the Nash equilibrium in the followers' Nash game is equivalently reformulated as the following VI by Proposition~\ref{prop:Nasheq.VI}:
\begin{align}\label{vi:follower}
    \langle G(x,y^*),y-y^*\rangle\ge0\quad\forall y\in Y(x),
\end{align}
where $y^*\in Y(x)$ denotes the Nash equilibrium.
In order to ensure the uniqueness of the Nash equilibrium in the followers' Nash game, i.e., the solution of VI~\eqref{vi:follower}, we further assume the following assumptions in this paper.
\begin{assumption}\label{asmp:followersVI}
    The Jacobian matrix of the mapping $G(x,\cdot)\colon\Re^m\to\Re^m$ is positive definite for any fixed $x$.
\end{assumption}

\begin{remark}\label{rem:assumption}
    Let us review the assumptions and problem settings used in the previous literature on MLMFG.
    Hu and Fukushima \cite{Hu2011} considered the multi-leader--single-follower quadratic game in which one follower solves the strictly convex quadratic optimization problem with linear equality constraints.
    In this setting, the optimality condition for the follower is necessary and sufficient, and hence the unique response can be analytically solved; in fact, the follower's response is linear.
    Herty et al. \cite{Herty2022} also considered the same quadratic game where the follower solves the strictly convex quadratic optimization with a positive diagonal matrix for the quadratic term but the constrains only consists of the componentwise lower bound, i.e., $y\ge l(x)$, where $l(x)$ is a linear function of $x$.
    In the setting, the follower's unique response is not smooth but can be solved analytically.
    They applied the smoothing method for follower's optimality conditions, and they then obtained the smoothed unique response even though the smoothing term is included.
    
    On the other hand, our settings can be seen as a generalization of theirs since we do not assume the detailed structure of $\gamma^\omega(x,\cdot,\cdot)$ or $Y(x)$.
    Note that the scope of both the literature above is to identify the existence or uniqueness of the LF Nash equilibrium of the MLMFG, but we do not consider the existence of LF Nash equilibrium, though the uniqueness of the Nash equilibrium of the followers' game is always ensured by Proposition~\ref{prop:existence.uniqueness.Follower.Nash}.
\end{remark}

\begin{proposition}\label{prop:existence.uniqueness.Follower.Nash}
    Suppose that \ref{asmp:follower1}--\ref{asmp:follower3} in Assumption~\ref{asmp:LF} and Assumption~\ref{asmp:followersVI} hold.
    Then the Nash equilibrium of followers' Nash game is unique.
\end{proposition}
\begin{proof}
    By the convexity of $\gamma^\omega(x,\cdot,y^{-\omega})$ and compactness of $Y^\omega(x)$ for all $\omega$, there exists a Nash equilibrium depending on $x\in X$ by Proposition~\ref{prop:existence.Nash.eq}.
    By the convexity of $Y^\omega(x)$ for all $\omega$, finding a Nash equilibrium is equivalent to solving VI \eqref{vi:follower}.
    Assumption~\ref{asmp:followersVI} implies that $G(x,\cdot)$ is strictly monotone for any fixed $x$.
    It then follows from Proposition~\ref{prop:uniquenss} that the solution to the strictly monotone variational inequality is at most one.
    Therefore the Nash equilibrium of the followers' Nash game uniquely exists.
\end{proof}

\begin{remark}
    If $Y(x)$ is not compact, the uniqueness still holds if the mapping $G(x,\cdot)$ is strongly monotone: There exists $\sigma>0$ such that
    $$
        \langle G(x,y)-G(x,y'),y-y'\rangle\ge\sigma\|y-y'\|^2\quad\forall y,y'\in Y(x).
    $$
\end{remark}

Omitting the follower's label $\omega$, we simplify the notations of the followers' constraint functions as follows:
$$
   g(x,y)=[g_i(x,y)]_{i=1}^l:=[g^\omega(x,y^\omega)]_{\omega=1}^M.
$$
Since $g^\omega(x,y^\omega)$ is independent of $y^{-\omega}$ and for any $y^\omega$ such that $g^\omega(x,y^\omega)\le0$, $y^\omega$ satisfies Assumption~\ref{asmp:LF}--\ref{asmp:follower3}, the LICQ for the collection of inequality constraints $g(x,y)\le0$ in VI \eqref{vi:follower} still holds.
Then the KKT conditions for the VI are written as follows:
\begin{align}\label{eq:KKT.followers}
\begin{aligned}
    G(x,y)+\nabla_y g(x,y)\lambda=0,\\
    g(x,y)+z=0,\\
    0\le\lambda\perp z\ge 0,
\end{aligned}
\end{align}
where $z\in\Re^{l}$ is a slack variable for the inequality constraint $g(x,y)\le0$, and $\lambda\in\Re^{l}$ represents the Lagrange multiplier.
If \eqref{eq:KKT.followers} is incorporated into the constraints of each leader's optimization problem \eqref{prob:leader.i-th}, the resultant problem is referred to as an EPEC.
Previous works such as \cite{Hori2019,Leyffer2010} have proposed solution methods for the EPEC associated with the MLMFG.

Now, using a Fischer--Burmeister function (FB-function) $\phi_0\colon\Re^2\to\Re$:
$$
    \phi_0(a,b):=\sqrt{a^2+b^2}-(a+b),
$$
the complementarity condition $0\le a\perp b\ge 0$ ($a\in\Re$, $b\in\Re$) is equivalent to $\phi_0(a,b)=0$.
Then, using this property, the complementarity $0\le\lambda\perp z\ge0$ is rewritten as
$$
    \Phi_0(\lambda,z):=\left[
    \begin{array}{c}
         \phi_0(\lambda_1,z_1) \\
         \vdots \\
         \phi_0(\lambda_l,z_l)
    \end{array}\right]=0.
$$
Let
$$
    H_0(x,y,z,\lambda):=
    \left[
        \begin{array}{c}
            G(x,y)+\nabla_y g(x,y)\lambda \\
            g(x,y)+z \\
            \Phi_0(\lambda,z)
        \end{array}
    \right].
$$
Then, KKT conditions \eqref{eq:KKT.followers} coincide with $H_0(x,y,z,\lambda)=0$.
Hereinafter, let $w:=(y,z,\lambda)$ and $H_0(x,w):=H_0(x,y,z,\lambda)$

\begin{proposition}\label{prop:VI.KKT}
    Given $x\in X$, let $w^*:=(y^*,z^*,\lambda^*)$ be the zero of the nonlinear equation $H_0(x,w)=0$.
    If \ref{asmp:follower1}--\ref{asmp:follower3} in Assumption~\ref{asmp:LF} and Assumption~\ref{asmp:followersVI} hold, then, $y^*$ is a Nash equilibrium of the followers' Nash game, and it is uniquely determined depending on $x\in X$.
\end{proposition}
\begin{proof}
    Since $w^*=(y^*,z^*,\lambda^*)$ satisfies $H_0(x,w^*)=0$, the tuple also satisfies KKT conditions \eqref{eq:KKT.followers} for VI \eqref{vi:follower}.
    It follows from the convexity of $Y(x)$ and \cite[Proposition 1.3.4]{Facchinei2003} that $y^*$ solves \eqref{vi:follower}, which implies $y^*$ is a Nash equilibrium of the followers' game by Proposition~\ref{prop:Nasheq.VI}.
    The uniqueness is ensured from Proposition~\ref{prop:existence.uniqueness.Follower.Nash}.%$G(x,\cdot)$ is strongly monotone on $Y(x)$ for arbitrary given $x\in X$, the solution is unique depending of $x\in X$.
\end{proof}

Given a tuple of leaders' strategies $x\in X$, we denote $w(x):=(y(x),z(x),\lambda(x))$ as a solution to $H_0(x,w)=0$.
Proposition~\ref{prop:VI.KKT} states that we can obtain a Nash equilibrium of followers' Nash game by solving $H_0(x,w)=0$.
Nevertheless, since $H_0$ is nonsmooth at which $z_i=\lambda_i=0$, \emph{degenerate point}, $y(x)$ is nonsmooth.
Consequently, reduced problem \eqref{prob:leader.i-th.reduced} is nonsmooth; it is numerically difficult to deal with.
To overcome it, we propose a smoothing approximation scheme for the equation.

Given a positive number $\varepsilon$, the \emph{smoothing} FB-function $\phi_\varepsilon\colon\Re^2\to\Re$ is defined as
$$
    \phi_\varepsilon(a,b):=\sqrt{a^2+b^2+2\varepsilon^2}-(a+b).
$$
It is easy to see that $\phi_\varepsilon$ is continuously differentiable everywhere, and $\phi_\varepsilon(a,b)\to\phi_0(a,b)$ ($\varepsilon\to 0$) by continuity.

Replacing $\phi_0$ with $\phi_\varepsilon$ in $H_0$, the perturbed nonlinear system is given by
$$
    H_\varepsilon(x,w)\equiv H_\varepsilon(x,y,z,\lambda) =0.
$$

Now we delve into some properties of $H_0$ and $H_\varepsilon$.

\begin{proposition}\label{prop:KKT.solution}
    Let $x\in X$ be fixed.
    For any $\varepsilon\ge0$, if \ref{asmp:follower1}--\ref{asmp:follower3} in Assumption~\ref{asmp:LF} and Assumption~\ref{asmp:followersVI} hold, then the system $H_\varepsilon(x,w)=0$ has a unique solution $w_\varepsilon(x):=(y_\varepsilon(x),z_\varepsilon(x),\lambda_\varepsilon(x))$, and $(z_\varepsilon(x),$ $\lambda_\varepsilon(x))$ satisfies $z_\varepsilon(x)>0$ and $\lambda_\varepsilon(x)>0$ with $[z_\varepsilon(x)]_i[\lambda_\varepsilon(x)]_i=\varepsilon^2$, where $[z_\varepsilon(x)]_i$ and $[\lambda_\varepsilon(x)]_i$ denote the $i$th element of the vectors $z_\varepsilon(x)$ and $\lambda_\varepsilon(x)$, respectively.
\end{proposition}
\begin{proof}
    It suffices to show the claim when $\varepsilon>0$ since we have proved the statement in the case where $\varepsilon=0$ in Proposition~\ref{prop:VI.KKT}.
    The solvability and uniqueness of the solution to $H_\varepsilon(x,w)=0$ is proved by Kanzow and Jiang \cite[Lemma 3.11]{Kanzow1998}.
    The latter statement is easily verified.
\end{proof}

We next show the nonsingularity of the (generalized) Jacobian matrix of $H_\varepsilon$ for any $\varepsilon\ge0$.

\begin{lemma}\label{lem:nonsingularity}
    Let $L\in\Re^{m\times m}$ be a (not necessarily symmetric) positive definite matrix and $A\in\Re^{m\times l}$ be arbitrary.
    Suppose that $\Xi\in\Re^{l\times l}$ and $H\in\Re^{l\times l}$ are diagonal matrices with negative entries.
    Then, the matrix
    $$
        M:=
        \left[
        \begin{array}{ccc}
            L & A & O \\
            O & I & \Xi \\
            A^\top & O & H 
        \end{array}
        \right]\in\Re^{(m+2l)\times (m+2l)}
    $$
    is nonsingular.
\end{lemma}
\begin{proof}
% Comment: The (almost) same result has already been proved in Lemma 3.4 by Kanzow and Jiang 1998.
    It suffices to show that the system of equation $Mv=0$ has only the trivial solution $v=0$.
    Let $v=(v_1,v_2,v_3)$, and then we have
    \begin{align}
        Lv_1+Av_2=0,\label{eq:lem:nonsingularity.1}\\
        v_2+\Xi v_3=0,\label{eq:lem:nonsingularity.2}\\
        A^\top v_1+H v_3=0.\label{eq:lem:nonsingularity.3}
    \end{align}
    It follows from \eqref{eq:lem:nonsingularity.1} that $v_1=-L^{-1}Av_2$.
    Since $\Xi$ is a negative diagonal matrix, $v_3=-\Xi^{-1}v_2$ in \eqref{eq:lem:nonsingularity.2}.
    Substituting them for \eqref{eq:lem:nonsingularity.3} yields
    $$
        (A^\top L^{-1}A+H\Xi^{-1})v_2=0.
    $$
    Since $A^\top L^{-1}A$ is positive semidefinite for any $A$, and $H\Xi^{-1}$ is a diagonal matrix whose diagonal entries are positive, the coefficient matrix $A^\top L^{-1}A+H\Xi^{-1}$ is positive definite.
    This implies $v_2=0$, and then $v_3=v_1=0$.
    We have completed the proof.
\end{proof}

\begin{lemma}\label{lem:nonsingularity.positive}
    Suppose that \ref{asmp:follower1}--\ref{asmp:follower3} in Assumption~\ref{asmp:LF} and Assumption~\ref{asmp:followersVI} hold.
    Let $x\in X$ and for $\varepsilon>0$, $w^*=(y^*,z^*,\lambda^*)$ be a solution to $H_\varepsilon(x,w)=0$.
    Then, the Jacobian matrix $\nabla_w H_\varepsilon$ is nonsingular.
\end{lemma}
\begin{proof}
    The Jacobian of $H_\varepsilon$ with respect to $w=(y,z,\lambda)$ is given as follows:
    $$
        \left[
        \begin{array}{ccc}
            L & A & O \\
            O & I & \Xi \\
            A^\top & O & H
        \end{array}
        \right],
    $$
    where
    \begin{align}
            L&:=\nabla_y G(x,y)+\sum_{i=1}^{l}\nabla^2_{yy}g_i(x,y)\lambda_i,\label{eq:prop:nonsingularity.Jacobian.1}\\
            A&:=\nabla_y g(x,y),\label{eq:prop:nonsingularity.Jacobian.2}\\
            \Xi&:=\underset{i=1,\dots,l}{\diag}\left(\frac{z_i}{\sqrt{(z_i)^2+(\lambda_i)^2+2\varepsilon^2}}-1\right)=\underset{i=1,\dots,l}{\diag}\left(\frac{z_i}{z_i+\lambda_i}-1\right),\nonumber\\
            H&:=\underset{i=1,\dots,l}{\diag}\left(\frac{\lambda_i}{\sqrt{(z_i)^2+(\lambda_i)^2+2\varepsilon^2}}-1\right)=\underset{i=1,\dots,l}{\diag}\left(\frac{\lambda_i}{z_i+\lambda_i}-1\right)\nonumber.
    \end{align}
    Here we use $z_i>0$, $\lambda_i>0$, and $z_i\lambda_i=\varepsilon^2$.
    It is obvious that $\Xi$ and $H$ consist of negative diagonal entries.
    Since $\nabla_y G(x,\cdot)$ is positive definite and $g(x,\cdot)$ is convex, $L$ is positive definite.
    Hence, applying Lemma~\ref{lem:nonsingularity} yields the result.
\end{proof}

For a given $x\in\Re^n$, we define the index sets below:
$$
    \begin{aligned}
        {\cal J}_{0+}(x)&:=\{i\mid z_i(x)=0<\lambda_i(x)\},\\
        {\cal J}_{00}(x)&:=\{i\mid z_i(x)=0=\lambda_i(x)\},\\
        {\cal J}_{+0}(x)&:=\{i\mid z_i(x)>0=\lambda_i(x)\},
    \end{aligned}
$$
where $z_i(x)$ and $\lambda_i(x)$ denote the $i$th element of $z(x)$ and $\lambda(x)$.

The following proposition is a slight modification of Theorem 3.5 and Lemma 3.12 in Kanzow and Jiang \cite{Kanzow1998}.
\begin{lemma}\label{lem:nonsingularity.0}
    Suppose that \ref{asmp:follower1}--\ref{asmp:follower3} in Assumption~\ref{asmp:LF} and Assumption~\ref{asmp:followersVI} hold.
    For a given $x\in X$, let $w^*$ be a solution to $H_0(x,w)=0$.
    Assume that the LICQ holds at $w^*$.
    Then, the generalized Jacobian matrix $\partial_{w}H_0(x,w^*)$ is nonsingular.
\end{lemma}
\begin{proof}
    The generalized Jacobian matrix of $H_0(x,w^*)$ with respect to $w$ is given by
    $$
    \begin{aligned}
        \partial_{w}H_0(x,w^*)={\Biggl\{}
        M=\left[
            \begin{array}{ccc}
            L & A & O \\
            O & I & \Xi \\
            A^\top & O & H
            \end{array}
        \right]\ {\Bigg|}\ 
        L=\eqref{eq:prop:nonsingularity.Jacobian.1}, A=\eqref{eq:prop:nonsingularity.Jacobian.2},\\
            \Xi=\underset{i=1,\dots,l}{\diag}(\xi_i-1),\ H=\underset{i=1,\dots,l}{\diag}(\eta_i-1).
        {\Biggr\}},
    \end{aligned}
    $$
    where
    $$
        \xi_i\in\left\{
        \begin{array}{ll}
        \{0\} & \text{if } i\in{\cal J}_{0+}(x) \\
        \left[0,1\right] & \text{if } i\in{\cal J}_{00}(x)\\
        \{1\} & \text{if } i\in{\cal J}_{+0}(x)
        \end{array}
        \right.,\quad
        \eta_i\in\left\{
        \begin{array}{ll}
        \{1\} & \text{if }i\in{\cal J}_{0+}(x) \\
        \left[0,1\right] & \text{if }i\in{\cal J}_{00}(x)\\
        \{0\} & \text{if }i\in{\cal J}_{+0}(x) 
        \end{array}
        \right.
    $$
    such that $\xi_i^2+\eta_i^2\le 1$ for all $i\in{\cal J}_{00}(x)$.
    
    It suffices to show the nonsingularity of $M$ for any $M\in\partial_{w}H_0(x,w^*)$.
    We show the the nonsingularity of $M^\top$ for convenience.
    Let $v=(v_1,v_2,v_3)$, and $M^\top v=0$ is given as follows:
    \begin{align}
        L^\top v_1+Av_3=0,\label{eq:prop:nonsingularity.genJac.1}\\
        A^\top v_1+v_2=0,\label{eq:prop:nonsingularity.genJac.2}\\
        \Xi v_2+Hv_3=0.\label{eq:prop:nonsingularity.genJac.3}
    \end{align}
    For $i\in{\cal J}_{0+}(x)$, $\Xi_i=-1$ and $H_{i}=0$.
    Then $[v_2]_i=0$ from \eqref{eq:prop:nonsingularity.genJac.3}.
    For $i\in{\cal J}_{00}(x)$, since either $\Xi_{i}$ or $H_{i}$ is negative, $[v_2]_i[v_3]_i\le0$ by \eqref{eq:prop:nonsingularity.genJac.3}.
    For $i\in{\cal J}_{+0}(x)$, $\Xi_i=0$ and $H_i=-1$.
    Then $[v_3]_i=0$ from \eqref{eq:prop:nonsingularity.genJac.3}.
    Summarizing these results yields $(v_2)^\top v_3\le0$.
    Then premultiplying \eqref{eq:prop:nonsingularity.genJac.2} with $v_3$ leads to
    $$
        \begin{aligned}
        &(v_3)^\top A^\top v_1+(v_3)^\top v_2=0\\
        \iff& (v_3)^\top A^\top v_1\ge0.
        \end{aligned}
    $$
    Furthermore, premultiplying \eqref{eq:prop:nonsingularity.genJac.1} with $v_1$ follows
    $$
        (v_1)^\top L^\top v_1+(v_1)^\top Av_3=0.
    $$
    Since $L^\top$ is positive definite and $(v_1)^\top Av_3\ge0$, $v_1=0$, which implies $v_2=0$ by \eqref{eq:prop:nonsingularity.genJac.2}.
    In \eqref{eq:prop:nonsingularity.genJac.1}, we have
    \begin{align}\label{eq:prop:nonsingularity.genJac.linearcomb}
        Av_3=0\iff\sum_{i=1}^l\nabla_y g_i(x,y)v_3=\sum_{i\in{\cal I}_g(x)}\nabla_y g_i(x,y)v_3=0,
    \end{align}
    where ${\cal I}_g(x):=\{i\mid g_i(x,y)=0,\ i=1,\dots,l\}$, and the last equality holds from $[v_3]_i=0$ for $i\in{\cal J}_{+0}(x)$.
    By the LICQ assumption and \eqref{eq:prop:nonsingularity.genJac.linearcomb}, $v_3=0$.
    Hence, we have $v=0$, and this implies that $M$ is nonsingular.
    This completes the proof.
\end{proof}

Summarizing the results of Lemmas~\ref{lem:nonsingularity.positive} and~\ref{lem:nonsingularity.0} yields the following proposition.

\begin{proposition}\label{prop:nonsingularity}
    Suppose that \ref{asmp:follower1}--\ref{asmp:follower3} in Assumption~\ref{asmp:LF} and Assumption~\ref{asmp:followersVI} hold.
    For every $\varepsilon\ge0$ and $x\in X$, the (generalized) Jacobian of $H_\varepsilon(x,\cdot)\colon\Re^{m+l+l}\to\Re^{m+l+l}$ is nonsingular. 
\end{proposition}

In what follows, we also use the notation $H(\varepsilon,x,w):=H_\varepsilon(x,w)$ to emphasize that $\varepsilon$ is one of the variables.
In the same way, $w(\varepsilon,x)\equiv w_\varepsilon(x)$ and
$
w(\varepsilon,x)=(y(\varepsilon,x), z(\varepsilon,x), \lambda(\varepsilon,x))\equiv (y_\varepsilon(x),z_\varepsilon(x),\lambda_\varepsilon(x)).$
%Let $\pi_{yz\lambda}$ be the projection operator over the $(y,z,\lambda)$ space.

\begin{lemma}\label{lem:Lipschitz.regular}
    Suppose that \ref{asmp:follower1}--\ref{asmp:follower3} in Assumption~\ref{asmp:LF} and Assumption~\ref{asmp:followersVI} hold.
    For every $\varepsilon\ge0$, $H(\varepsilon,x,w)$, as a function of the variables $(\varepsilon,x,w)$, is locally Lipschitz continuous and regular.
\end{lemma}
\begin{proof}
    Since~\ref{asmp:follower1} holds, and thus all the remaining components of $H(\varepsilon,x,w)$ except the FB-function $\phi_\varepsilon$ when $\varepsilon=0$ are continuously differentiable from \ref{asmp:follower1}, we only need to show that the locally Lipschitz continuity and regularity of $\phi_0$.
    It is obvious that $\phi_0$ is convex by its definition.
    It follows from \cite[Proposition 2.3.6-(b)]{Clarke1990} that $\phi_0$ is regular, and also $\phi_0$ is locally Lipschitz continuous whenever $\lambda_i$ and $z_i$ are bounded, where the boundedness of $\lambda_i$ and $z_i$ is satisfied from Assumptions~\ref{asmp:follower2} and~\ref{asmp:follower3}.
\end{proof}

\begin{proposition}\label{prop:implicit.y.z.lmd}
    Let $(\varepsilon,x,w)$ be such that $H(\varepsilon,x,w)=0$.
    If \ref{asmp:follower1}--\ref{asmp:follower3} in Assumption~\ref{asmp:LF} and Assumption~\ref{asmp:followersVI} hold, then there is a neighborhood $U\times\Omega\subset\Re^{1+n}$ of $(\varepsilon,x)$ and a locally Lipschitz continuous function $w\colon U\times\Omega\to\Re^{m+l+l}$ such that for each $(\varepsilon,x)\in U\times\Omega$,
    $$
        H(\varepsilon,x,w(\varepsilon,x))=0.
    $$
    Moreover, for any fixed $\varepsilon\in U\setminus\{0\}$, $w_\varepsilon\colon\Omega\to\Re^{m+l+l}$ is continuously differentiable.
\end{proposition}
\begin{proof}
    By Lemma~\ref{lem:Lipschitz.regular}, Proposition~\ref{prop:nonsingularity}, the implicit function theorem \cite[Corollary to Theorem 7.1.1]{Clarke1990}, and \cite[Lemma~2]{Facchinei1999} yields that $y,z,\lambda$ are locally Lipschitz continuous on $\Omega$.
    The latter claim is obtained from a well-known result in elementary calculus.
\end{proof}

Note that the local Lipschitz continuity of $w(\cdot,\cdot)$ from Proposition~\ref{prop:implicit.y.z.lmd} implies that $w_\varepsilon(x)\to w(x^*)$, i.e.,
$$
    y_\varepsilon(x)\to y(x^*),\ z_\varepsilon(x)\to z(x^*),\ \lambda_\varepsilon(x)\to\lambda(x^*),
$$
as $x\to x^*$ and $\varepsilon\downarrow0$, and for any $\varepsilon>0$. 
Now we show some properties of $w(\cdot,\cdot)$ and $w_\varepsilon(\cdot)$.

By the compactness assumption of $X\subset\Re^n$ and continuity assumption, the following statement holds, which is derived from the elementary calculus.
%The image of a compact set for the continuous function is also compact.
\begin{proposition}\label{prop:boundedness.grad}
    Under \ref{asmp:follower1}--\ref{asmp:follower3} in Assumption~\ref{asmp:LF}, if \ref{asmp:leader2} in Assumption~\ref{asmp:LF} holds, then the function $w\colon\Re^{1+n}\to\Re^{m+l+l}$ is compact-valued over $X$, and for any fixed $\varepsilon\ge 0$, its partial (generalized) Jacobian matrix $\partial_x w_\varepsilon$ is also compact.
\end{proposition}

\begin{lemma}\label{lem:convergence.degenerate}
    For a positive sequence $\{\varepsilon_k\}$ converging to $0$, let $a_{\varepsilon_k}>0$ and $b_{\varepsilon_k}>0$ for all $k$ and converging to $0$.
    Suppose that
    $$
        \hat{\xi}^k:=\frac{a_{\varepsilon_k}}{\sqrt{a_{\varepsilon_k}^2+b_{\varepsilon_k}^2+2\varepsilon_k^2}},\ 
        \hat{\eta}^k:=\frac{b_{\varepsilon_k}}{\sqrt{a_{\varepsilon_k}^2+b_{\varepsilon_k}^2+2\varepsilon_k^2}}.
    $$
    Then their limits, if they exist, are denoted by $\hat{\xi}^\circ$ and $\hat{\eta}^\circ$, respectively, and satisfy $\hat{\xi}^\circ,\hat{\eta}^\circ\in[0,1]$ and $(\hat{\xi}^\circ)^2+(\hat{\eta}^\circ)^2\le1$.
\end{lemma}
\begin{proof}
    Let $a_{\varepsilon_k}=r_k\cos\theta_k$ and $b_{\varepsilon_k}=r_k\sin\theta_k$, where $r_k\to0$.
    Then we have
    \[
        \hat{\xi}^k=\frac{\cos\theta_k}{\sqrt{1+2(\varepsilon_k/r_k)^2}},\ \hat{\eta}^k=\frac{\sin\theta_k}{\sqrt{1+2(\varepsilon_k/r_k)^2}}.
    \]
    Obviously, their limits $\hat{\xi}^\circ,\hat{\eta}^\circ$ satisfy $\hat{\xi}^\circ,\hat{\eta}^\circ\in[0,1]$.
    Furthermore,
    \[
        (\hat{\xi}^k)^2+(\hat{\eta}^k)^2=\frac{1}{1+2(\varepsilon_k/r_k)^2}<1,
    \]
    for all $k=1,2,\dots$.
    Therefore, $(\hat{\xi}^\circ)^2+(\hat{\eta}^\circ)^2\le1$.
\end{proof}

\begin{proposition}\label{prop:response.function}
    Under \ref{asmp:follower1}--\ref{asmp:follower3} in Assumption~\ref{asmp:LF} and Assumption~\ref{asmp:followersVI}, there is a positive sequence $\{\varepsilon_k\}$ tending to $0$ such that
    \begin{align}\label{ieq:prop:response.inclusion}
        \left\{\lim_{k\to\infty}\nabla w_{\varepsilon_k}(x^k)\right\}\subset\partial w(x^*).
    \end{align}
\end{proposition}
\begin{proof}
    By the local Lipschitz continuity of $w(\cdot)$, the generalized Jacobian $\partial w(x)$ is
    $$
        \partial w(x^*)=\co\left\{\lim_{x\to x^*}\nabla w(x)\ \middle|\ x\in{\cal D}\right\},
    $$
    where ${\cal D}\subset\Re^n$ denotes the set of points at which $w$ is differentiable; to be exact, that of points at which $y$ and $\lambda$ are differentiable because 
    if $y(\cdot)$ is differentiable at $x^*$, $z(\cdot)$ is also differentiable at $x^*$ from the definition of $z$.
    Hence, it suffices to show, instead of \eqref{ieq:prop:response.inclusion}, that
    \begin{align}\label{ieq:inclusion.equiv}
        \left\{\lim_{k\to\infty}\nabla w_{\varepsilon_k}(x^k)\right\}\subset\co\left\{\lim_{x\to x^*}\nabla w(x)\ \middle|\ x\in{\cal D}\right\}.
    \end{align}

    For $\varepsilon_k>0$, Proposition~\ref{prop:implicit.y.z.lmd} leads that the gradient of $w_{\varepsilon_k}$ at $x^k$ is given by
    \begin{align}\label{eq:prop:response.derivative}
        \nabla w_{\varepsilon_k}(x^k)&=\left[\nabla y_{\varepsilon_k}(x^k),\ \nabla z_{\varepsilon_k}(x^k),\ \nabla\lambda_{\varepsilon_k}(x^k)\right]\nonumber\\
        &=-
        \nabla_x H_{\varepsilon_k}(x^k,y_{\varepsilon_k}(x^k),z_{\varepsilon_k}(x^k),\lambda_{\varepsilon_k}(x^k))\nonumber\\
        &\qquad\qquad \left[
        \begin{array}{c}
            \nabla_y H_{\varepsilon_k}(x^k,y_{\varepsilon_k}(x^k),z_{\varepsilon_k}(x^k),\lambda_{\varepsilon_k}(x^k)) \\
            \nabla_z H_{\varepsilon_k}(x^k,y_{\varepsilon_k}(x^k),z_{\varepsilon_k}(x^k),\lambda_{\varepsilon_k}(x^k)) \\
            \nabla_\lambda H_{\varepsilon_k}(x^k,y_{\varepsilon_k}(x^k),z_{\varepsilon_k}(x^k),\lambda_{\varepsilon_k}(x^k)) 
        \end{array}
        \right]^{-1}\nonumber\\
        &=
        -\left[ {L'_{\varepsilon_k}},\ {A'_{\varepsilon_k}},\ O \right]
        \left[
        \begin{array}{ccc}
            L_{\varepsilon_k} & A_{\varepsilon_k} & O \\
            O &  I & \Xi_{\varepsilon_k} \\
            A_{\varepsilon_k}^\top & O & H_{\varepsilon_k}
        \end{array}
        \right]^{-1},
    \end{align}
    where
    \[
    \begin{aligned}
    L_{\varepsilon_k}&:=\nabla_y G(x^k,y_{\varepsilon_k}(x^k))+\sum_{i=1}^l[\lambda_{\varepsilon_k}(x^k)]_i\nabla_{yy}^2 g_i(x^k,y_{\varepsilon_k}(x^k)),\\
    A_{\varepsilon_k}&:=\nabla_y g(x^k,y_{\varepsilon_k}(x^k)),\\
    L'_{\varepsilon_k}&:=\nabla_x G(x^k,y_{\varepsilon_k}(x^k))+\sum_{i=1}^{l}[\lambda_{\varepsilon_k}(x^k)]_i\nabla^2_{xy}g_i(x^k,y_{\varepsilon_k}(x^k)),\\
    A'_{\varepsilon_k}&:=\nabla_x g(x^k,y_{\varepsilon_k}(x^k)),\\
    \Xi_{\varepsilon_k}&:=\underset{i=1,\dots,l}{\diag}\left(\frac{[z_{\varepsilon_k}(x^k)]_i}{\sqrt{[z_{\varepsilon_k}(x^k)]_i^2+[\lambda_{\varepsilon_k}(x^k)]_i^2+2\varepsilon_k^2}}-1\right)\\
    &=\underset{i=1,\dots,l}{\diag}\left(\frac{[z_{\varepsilon_k}(x^k)]_i}{[z_{\varepsilon_k}(x^k)]_i+[\lambda_{\varepsilon_k}(x^k)]_i}-1\right),\\
    H_{\varepsilon_k}&:=\underset{i=1,\dots,l}{\diag}\left(\frac{[\lambda_{\varepsilon_k}(x^k)]_i}{\sqrt{[z_{\varepsilon_k}(x^k)]_i^2+[\lambda_{\varepsilon_k}(x^k)]_i^2+2\varepsilon_k^2}}-1\right)\\
    &=\underset{i=1,\dots,l}{\diag}\left(\frac{[\lambda_{\varepsilon_k}(x^k)]_i}{[z_{\varepsilon_k}(x^k)]_i+[\lambda_{\varepsilon_k}(x^k)]_i}-1\right).
    \end{aligned}
    \]
    Here, the second equality in $\Xi_{\varepsilon_k}$ and $H_{\varepsilon_k}$ holds from Proposition~\ref{prop:KKT.solution} and since $\varepsilon>0$.
    Let
    \[
        \hat{\xi}^k_i:=\frac{[z_{\varepsilon_k}(x^k)]_i}{\sqrt{[z_{\varepsilon_k}(x^k)]_i^2+[\lambda_{\varepsilon_k}(x^k)]_i^2+2\varepsilon_k^2}},\ 
        \hat{\eta}^k_i:=\frac{[\lambda_{\varepsilon_k}(x^k)]_i}{\sqrt{[z_{\varepsilon_k}(x^k)]_i^2+[\lambda_{\varepsilon_k}(x^k)]_i^2+2\varepsilon_k^2}},
    \]
    and the limits of $\hat{\xi}^k$ and $\hat{\eta}^k$ be $\hat{\xi}^\circ$ and $\hat{\eta}^\circ$, respectively.
    Since $(z_{\varepsilon_k},\lambda_{\varepsilon_k})$ is compact-valued on $X$ from Proposition~\ref{prop:boundedness.grad}, there exists a limit for appropriately chosen $\{\varepsilon_k\}$ and using Lemma~\ref{lem:convergence.degenerate}, we have
    $$
        \bar{\Xi}^\circ_i=\left\{
        \begin{array}{ll}
            -1 & \text{if } i\in{\cal J}_{0+}(x^*),\\
            \hat{\xi}^\circ_i-1 & \text{if } i\in{\cal J}_{00}(x^*),\\
            0 & \text{if } i\in{\cal J}_{+0}(x^*),
        \end{array}
        \right.\quad
        \bar{H}^\circ_i=\left\{
        \begin{array}{ll}
            0 & \text{if } i\in{\cal J}_{0+}(x^*),\\
            \hat{\eta}^\circ_i-1 & \text{if } i\in{\cal J}_{00}(x^*),\\
            -1 & \text{if } i\in{\cal J}_{+0}(x^*),
        \end{array}
        \right.
    $$
    where $\hat{\xi}^\circ_i,\hat{\eta}^\circ_i\in[0,1]$ such that $(\hat{\xi}^\circ_i)^2+(\hat{\eta}^\circ_i)^2\le 1$.
    In addition, since $G(\cdot)$ and $g(\cdot)$ are smooth, and $(y,z,\lambda)(\cdot,\cdot)$ is continuous, we can take the limits $L_{\varepsilon_k}\to\bar{L}$, $A_{\varepsilon_k}\to\bar{A}$, $L'_{\varepsilon_k}\to\bar{L}'$, and $A'_{\varepsilon_k}\to\bar{A}'$.
    Letting $\nabla w^\circ(x^*):=(\nabla y^\circ(x^*),\nabla z^\circ(x^*),\nabla \lambda^\circ(x^*))$ be the limit of \eqref{eq:prop:response.derivative}, we have
    $$
        \nabla w^\circ(x^*)^\top=
        \left[
        \begin{array}{c}
             \nabla y^\circ(x^*)^\top \\
             \nabla z^\circ(x^*)^\top \\
             \nabla \lambda^\circ(x^*)^\top
        \end{array}
        \right]
        =
        -\left[
        \begin{array}{ccc}
            \bar{L}^\top & O & \bar{A} \\
            \bar{A}^\top & I & O \\
            O & \bar{\Xi}^\circ & \bar{H}^\circ
        \end{array}
        \right]^{-1}
        \left[
        \begin{array}{c}
            {\bar{L'}}^\top \\
            {\bar{A'}}^\top \\
            O
        \end{array}\right].
    $$
    
    Now we show that
    \begin{align}\label{show.belonging}
        \nabla w^\circ(x^*) \in
        \co\left\{\lim_{x\to x^*,x\in{\cal D}}\nabla w(x)\right\}.
    \end{align}
    for $\{\varepsilon_k\}$ appropriately chosen.

    We first consider the case where $x^*\in{\cal D}$.
    Let
    $$
        T(x):=H_0(x,w(x))\equiv0\quad\forall x\in X,
    $$
    because $w(x)$ is the solution to the system $H_0(x,w)=0$ for a fixed $x\in X$.
    For a given $x^k\in{\cal D}$, by differentiating both sides of $T(x)\equiv0$, we have
    $$
        \nabla w(x^k)^\top=
        \left[
        \begin{array}{c}
            \nabla y(x^k)^\top \\
            \nabla z(x^k)^\top \\
            \nabla \lambda(x^k)^\top
        \end{array}
        \right]
        =-\left[
        \begin{array}{ccc}
            {L_k}^\top & O & A_k \\
            A_k^\top & I & O \\
            O & \Xi_k & H_k
        \end{array}
        \right]^{-1}
        \left[
        \begin{array}{c}
            {L'_k}^\top \\
            {A'_k}^\top \\
            O
        \end{array}
        \right],
    $$
    where
    $$
        \begin{aligned}
        L_k&:=\nabla_y G(x^k,y(x^k))+\sum_{i=1}^l{\lambda_i}(x^k)\nabla_{yy}^2 g_i(x^k,y(x^k)),\\
        A_k&:=\nabla_y g(x^k,y(x^k)),\\
        L'_k&:=\nabla_x G(x^k,y(x^k))+\sum_{i=1}^{l}{\lambda_i}(x^k)\nabla^2_{xy}g_i(x^k,y(x^k)),\\
        A'_k&:=\nabla_x g(x^k,y(x^k)),\\
        \Xi_k&:=\underset{i=1,\dots,l}{\diag}\left(\frac{{z_i}(x^k)}{\sqrt{{z_i}(x^k)^2+{\lambda_i}(x^k)^2}}-1\right),\\
        H_k&:=\underset{i=1,\dots,l}{\diag}\left(\frac{{\lambda_i}(x^k)}{\sqrt{{z_i}(x^k)^2+{\lambda_i}(x^k)^2}}-1\right),
        \end{aligned}
    $$
    Note that letting
    $$
        \xi^k_i:=\frac{{z_i}(x^k)}{\sqrt{{z_i}(x^k)^2+{\lambda_i}(x^k)^2}},\quad\eta^k_i:=\frac{{\lambda_i}(x^k)}{\sqrt{{z_i}(x^k)^2+{\lambda_i}(x^k)^2}},
    $$
    where $\xi^k_i,\eta^k_i\in[0,1]$ such that $(\xi^k_i)^2+(\eta^k_i)^2=1$ yields
    $$
        (\Xi_k)_i=\left\{
        \begin{array}{ll}
            -1 & \text{if }i\in{\cal J}_{0+}(x^k),\\
            \xi^k_i-1 & \text{if }i\in{\cal J}_{00}(x^k),\\
            0 & \text{if }i\in{\cal J}_{+0}(x^k),
        \end{array}
        \right.\quad
        (H_k)_i=\left\{
        \begin{array}{ll}
            0 & \text{if }i\in{\cal J}_{0+}(x^k),\\
            \eta^k_i-1 & \text{if }i\in{\cal J}_{00}(x^k),\\
            -1 & \text{if }i\in{\cal J}_{+0}(x^k).
        \end{array}
        \right.
    $$
    Now as $k\to\infty$, i.e., $x^k\to x^*$, by the continuity of the involved functions, we have
    \begin{align}\label{eq:limit.diff.response}
        \nabla w(x^*)^\top=
        \left[
        \begin{array}{c}
            \nabla y(x^*)^\top \\
            \nabla z(x^*)^\top \\
            \nabla \lambda(x^*)^\top
        \end{array}
        \right]
        =
        -\left[
        \begin{array}{ccc}
            \bar{L}^\top & O & \bar{A} \\
            \bar{A}^\top & I & O \\
            O & \bar{\Xi} & \bar{H}
        \end{array}
        \right]^{-1}
        \left[
        \begin{array}{c}
            {\bar{L'}}^\top \\
            {\bar{A'}}^\top \\
            O
        \end{array}\right],
    \end{align}
    where 
    \begin{align}\label{eq:FB.limit}
        \bar{\Xi}_i=\left\{
        \begin{array}{ll}
            -1 & \text{if }i\in{\cal J}_{0+}(x^*), \\
            \bar{\xi}_i-1 & \text{if }i\in{\cal J}_{00}(x^*),\\
            0 & \text{if }i\in{\cal J}_{+0}(x^*),
        \end{array}
        \right.\quad
        \bar{H}_i=\left\{
        \begin{array}{ll}
            0 & \text{if }i\in{\cal J}_{0+}(x^*), \\
            \bar{\eta}_i-1 & \text{if }i\in{\cal J}_{00}(x^*),\\
            -1 & \text{if }i\in{\cal J}_{+0}(x^*),
        \end{array}
        \right.
    \end{align}
    for $\bar{\xi}_i,\bar{\eta}_i\in[0,1]$ such that $(\bar{\xi}_i)^2+(\bar{\eta}_i)^2=1$.
    Now taking the convex hull of the set that consists of \eqref{eq:limit.diff.response} yields that
    $$
        \begin{aligned}
        &\co\left\{\lim_{\substack{x\to x^*\\x\in{\cal D}}}\nabla w(x)\right\}=\\
        &\left\{
        \nabla w(x^*)
        \ \middle|\ 
        \text{\eqref{eq:limit.diff.response} and \eqref{eq:FB.limit} for $\bar{\xi}_i,\bar{\eta}_i\in[0,1]$ s.t. $\bar{\xi}_i^2+\bar{\eta}_i^2\le1$ for all $i\in{\cal J}_{00}(x^*)$.}
        \right\}
        \end{aligned}
    $$
    From the observation above, we can choose $\bar{\xi}_i$ and $\bar{\eta}_i$ so that \eqref{show.belonging} holds.
    
    Next, we consider the case where $x^*\notin{\cal D}$.
    For $x^k\in{\cal D}$, by the continuity of $\nabla w_{\varepsilon_k}$, we have
    $$
        \lim_{\varepsilon\downarrow0}\nabla w_{\varepsilon}(x^k)-\nabla w(x^k)=0.
    $$
    Hence for $x_k\in{\cal D}$, there exists $\varepsilon_k>0$ such that
    $$
        \|\nabla w_{\varepsilon_k}(x^k)-\nabla w(x^k)\|\le\|x_k-x^*\|,
    $$
    since the value of the left-hand side of the above inequality is sufficiently close to 0 for $\varepsilon_k$ sufficiently small.
    This implies that $\|\nabla w_{\varepsilon_k}(x^k)-\nabla w(x^k)\|\to 0$ as $x_k\to x^*$ and $\varepsilon_k\to 0$.
    %Here, Proposition~\ref{prop:boundedness.grad} leads to the boundedness of $\{\nabla(y,z,\lambda)(x^k)\}$ and $\{\nabla(y_{\varepsilon_k},z_{\varepsilon_k},\lambda_{\varepsilon_k})(x^k)\}$.
    %Since $\{\nabla(y,z,\lambda)(x^k)\}$ and $\{\nabla(y_{\varepsilon_k},z_{\varepsilon_k},\lambda_{\varepsilon_k})(x^k)\}$ respectively have limits as we discussed above,
    Since $\{\nabla w_{\varepsilon_k}(x^k)\}$ has a limit, we have
    $$
    \begin{aligned}
        \lim_{k\to\infty}\nabla w(x^k)&=
        \lim_{k\to\infty}[\nabla w_{\varepsilon_k}(x^k)-
        \{\nabla w_{\varepsilon_k}(x^k)-\nabla w(x^k)\}]\\
        &=\lim_{k\to\infty}\nabla w_{\varepsilon_k}(x^k) - 
        \lim_{k\to\infty}\{\nabla w_{\varepsilon_k}(x^k)-\nabla w(x^k)\}\\
        &=\nabla w^\circ(x^*)- 0 =\nabla w^\circ(x^*)
    \end{aligned}
    $$
    Hence it follows that
    $$
        \nabla w^\circ(x^*)=
        \lim_{k\to\infty}\nabla w(x^k)\in\co\left\{\lim_{\substack{x\to x^*\\x\in{\cal D}}}\nabla w(x)\right\}.
    $$
    We have thus completed the proof.
\end{proof}

\begin{corollary}\label{cor:y.inclusion}
    Under the same assumption as Proposition~\ref{prop:response.function}, for a positive sequence $\{\varepsilon_k\}$ that satisfies \eqref{ieq:prop:response.inclusion}, the following inclusion also holds:
    $$
        \left\{\lim_{k\to\infty}\nabla y_{\varepsilon_k}(x^k)\right\}\subset\partial y(x^*).
    $$
\end{corollary}
\begin{proof}
    By Proposition~\ref{prop:response.function}, we have
    $$
        \lim_{k\to\infty}\nabla w_{\varepsilon_k}(x^k)=
        \nabla w^\circ(x^*),
    $$
    and \cite[Proposition 2.6.2]{Clarke1990} leads to
    $$
        \partial w(x^*)\subset\partial y(x^*)\times\partial z(x^*)\times\partial \lambda(x^*).
    $$
    It then follows that
    $$
        \nabla w^\circ(x^*)=[\nabla y^\circ(x^*),\nabla z^\circ(x^*),\nabla \lambda^\circ(x^*)]\in\partial y(x^*)\times\partial z(x^*)\times\partial \lambda(x^*)
    $$
    
\end{proof}

By the observation from Corollary~\ref{cor:y.inclusion}, it may be reasonable to assume that
\begin{align}\label{eq:inclusion.i}
    \left\{\lim_{k\to\infty}\nabla_{x^\nu}y_{\varepsilon_k}(x^{k,\nu},x^{k,-\nu})\right\}\subset\partial_{x^\nu}y(x^{*,\nu},x^{*,-\nu}).
\end{align}

\begin{remark}
    Obviously, the relation \eqref{eq:inclusion.i} holds when $N=1$, i.e., the single-leader--multi-follower game.
\end{remark}

%{\color{red}
%\begin{theorem}\label{thm:partial.Jacobian.inclusion}
%    If $y(\cdot)$ is regular at $x^*\in X$, then
%    $$
%        \left\{\lim_{k\to\infty}\nabla_{x^\nu}y_{\varepsilon_k}(x^{k,\nu},x^{k,-\nu})\right\}\subset\partial_{x^\nu}y(x^{*,\nu},x^{*,-\nu}).
%    $$
%\end{theorem}
%\begin{proof}
%    First, let $\pi_\nu$ denote the projection operator onto $\Re^{n_\nu\times m}$ defined as follows:
%    $$
%        \pi_\nu\partial y(x^*):=\left\{V^\nu\in\Re^{n_\nu\times m}\ \middle|\ \text{for some $V^{\nu'}\in\Re^{n_{\nu'}\times m}$ s.t. $\nu'\neq\nu$}, (V^1,\dots,V^N)\in\partial y(x^*)\right\}.
%    $$
%    By the regularity assumption of $y$, applying \cite[Proposition 2.3.15]{Clarke1990} leads to
%    $$
%        \pi_\nu\partial y(x^*)\subset\partial_{x^\nu}y(x^{*,\nu},x^{*,-\nu}).
%    $$
%    If 
%    $$
%        \left\{\lim_{k\to\infty}\nabla_{x^\nu}y_{\varepsilon_k}(x^{k,\nu},x^{k,-\nu})\right\}\subset\pi_\nu\partial y(x^*)
%    $$
%    or
%    $$
%        \left\{\lim_{k\to\infty}\nabla_{x^\nu}y_{\varepsilon_k}(x^{k,\nu},x^{k,-\nu})\right\}\subset\pi_\nu\left\{\lim_{k\to\infty}\nabla y_{\varepsilon_k}(x^k)\right\},
%    $$
%    
%\end{proof}
%}
%because showing the inclusion above is technically difficult in general.

Herty et al. \cite{Herty2022} proposed a smoothing method for a special case of multi-leader--single-follower game in which the follower's response $y(x)$ can explicitly be obtained, while we consider the case where the response cannot be written explicitly in general.
Nevertheless, the gradient of the response function can be computed as indicated in~\eqref{eq:prop:response.derivative}.
%In practice, the computation of $\nabla_x y_\varepsilon(x)$ is based on a conjugate gradient method\footnote{In the literature on bilevel optimization in machine learning, the way of computation of the hypergradient $\nabla_x y_\varepsilon(x)$ is called an \emph{approximate implicit differentiation} (AID) method. Meanwhile, another way to obtain $\nabla_x y_\varepsilon(x)$ is called an \emph{iterative differentiation} (ITD) method but this may not be applicable for constrained optimization in the followers' problems.}.

By Proposition~\ref{prop:implicit.y.z.lmd}, reduced problem \eqref{prob:leader.i-th.reduced} may be approximated by a differentiable optimization problem:
\begin{align}\label{prob:approximated.reduced.leader.i}
    \min_{x^\nu\in X^\nu}\Theta^\nu_\varepsilon(x^\nu,x^{-\nu}):=\theta^\nu(x^\nu,x^{-\nu},y_\varepsilon(x^\nu,x^{-\nu})).
\end{align}
Now let {\NEPeps} be the game in which leader $\nu\in\{1,\dots,N\}$ solves \eqref{prob:approximated.reduced.leader.i}.
If {\NEPeps} has a Nash equilibrium, i.e., there exists $x^*$ such that
$$
    x^{*,\nu}\in\arg\min_{x^\nu\in X^\nu}\Theta^\nu_\varepsilon(x^\nu,x^{*,-\nu}),
$$
then the following assertion holds.
\begin{theorem}\label{thm:convergence.Nash}
    Let a sequence $\{x^k\}_{k\in\mathbb{N}}$ be such that each $x^k$ is a Nash equilibrium of {\NEPepsk}.
    If $\{x^k\}_{k\in\mathbb{N}}$ converges to $x^*$, then $x^*$ is a Nash equilibrium of {\NEP}.
    Moreover, $(x^*,y(x^*))$ is an LF Nash equilibrium of the MLMFG.
\end{theorem}
\begin{proof}
    The former assertion can be shown using \cite[Theorem 4.5]{Hu2012}, and the latter one holds from Proposition~\ref{prop:equivalence.GNEP.MLF}.
\end{proof}

However, analogous to the claim for {\NEP} before, since \eqref{prob:approximated.reduced.leader.i} is also not convex in general, the Nash equilibrium of {\NEPeps} may not exist.
Hence, we introduce the following stationary equilibrium concept for the nonconvex NEP.
\begin{definition}\label{def:stationary.NE}
    A tuple of strategies $x^*\in X$ is referred to as a \emph{stationary Nash equilibrium} of {\NEPeps} if $x^{*,\nu}$ satisfies the following condition for all $\nu\in\{1,\dots,N\}$:
    \begin{align}\label{eq:stationary.NE}
        (\Theta^\nu_\varepsilon)'(x^{*,\nu},x^{*,-\nu};d^\nu)=\langle\nabla_{x^\nu}\Theta^\nu_\varepsilon(x^{*,\nu},x^{*,-\nu}),d^\nu\rangle\ge0\quad\forall d^\nu\in{\cal T}_{X^\nu}(x^{*,\nu}),
    \end{align}
    which is equivalent to
    $$
    \begin{aligned}
        \nabla_{x^\nu}&\theta^\nu(x^{*,\nu},x^{*,-\nu},y_\varepsilon(x^{*,\nu},x^{*,-\nu}))^\top d^\nu+\\
        &{d^\nu}^\top\nabla_{x^\nu}y_\varepsilon(x^{*,\nu},x^{*,-\nu})\nabla_y\theta^\nu(x^{*,\nu},x^{*,-\nu},y_\varepsilon(x^{*,\nu},x^{*,-\nu}))\ge0\quad\forall d^\nu\in{\cal T}_{X^\nu}(x^{*,\nu}).
    \end{aligned}
    $$
    %where
    %$$
    %    F_\varepsilon(x):=\left[
    %    \begin{array}{c}
    %        \nabla_{x^1}\Theta_1(x^1,x^{-1};\varepsilon)\\
    %        \vdots\\
    %        \nabla_{x^N}\Theta_N(x^N,x^{-N};\varepsilon)\\
    %    \end{array}
    %    \right].
    %$$
\end{definition}
Definition~\ref{def:stationary.NE} implies that for all $\nu$, $x^{*,\nu}\in X^\nu$ is a stationary point of reduced problem \eqref{prob:approximated.reduced.leader.i} for fixed $x^{*,-\nu}$.
Note that since $\Theta^\nu$ is differentiable, \eqref{eq:stationary.NE} is equivalent to
$$
    0\in\nabla_{x^\nu}\Theta^\nu_\varepsilon(x^{*,\nu},x^{*,-\nu})+{\cal T}_{X^\nu}(x^{*,\nu})^\circ,
$$
which means that the B-/C-stationary Nash equilibrium, introduced in Definitions~\ref{def:stationary.Nash} and~\ref{def:weak.stationary.Nash}, are equivalent under the differentiability of $\Theta^\nu_\varepsilon$.
%If $X^\nu\subset\Re^{n_\nu}$ is regular, i.e., ${\cal T}_{X^\nu}(x^{*,\nu})^*={\cal N}_{X^\nu}(x^{*,\nu})$, \eqref{eq:stationary.NE} is equivalent to
%%By \cite[Proposition 4.6.3]{Bertsekas2003} under the convexity of $X^\nu$, \eqref{eq:stationary.NE} is equivalent to
%\begin{align}\label{eq:stationary.NE.normal}
%    0\in\nabla_{x^\nu}\Theta^\nu(x^{*,\nu},x^{*,-\nu};\varepsilon)+{\cal N}_{X^\nu}(x^{*,\nu}).
%    %&\Theta'_i(x^{*,\nu},x^{*,-\nu};\varepsilon,d^\nu)\ge0\quad\forall d^\nu\in{\cal T}_{X^\nu}(x^{*,\nu})\nonumber\\
%    %\iff&\langle\nabla_{x^\nu}\theta^\nu(x^{*,\nu},x^{*,-\nu},y_\varepsilon(x^{*,\nu},x^{*,-\nu})),d^\nu\rangle\ge0\quad\forall d^\nu\in{\cal T}_{X^\nu}(x^{*,\nu}).
%\end{align}

In practice, we sequentially obtain a stationary Nash equilibrium of {\NEPeps} as $\varepsilon>0$ decreases to find an approximate B-/C-stationary Nash equilibrium for {\NEP}.

If $X^\nu$ is convex for all $\nu$, the equilibrium can be computed by solving the following variational inequality problem: Find $x^*\in X$ such that
\begin{align}\label{ieq:VI.stationary.Nash}
    \langle F^\ell_\varepsilon(x^*), x-x^*\rangle\ge0\qquad\forall x\in X,
\end{align}
where
$$
    F^\ell_\varepsilon(x):=
    \left[
    \begin{array}{c}
        \nabla_{x^1}\Theta^1_\varepsilon(x^1,x^{-1}) \\
        \vdots \\
        \nabla_{x^N}\Theta^N_\varepsilon(x^N,x^{-N}) 
    \end{array}
    \right].
$$
The following proposition guarantees that the solution to VI \eqref{ieq:VI.stationary.Nash} is the stationary Nash equilibrium for {\NEPeps}.
\begin{proposition}\label{prop:equivalence.VI}
    Suppose that $X^\nu\subset\Re^{n_\nu}$ is convex for all $\nu$.
    Let $x^*\in X$ be a solution to \eqref{ieq:VI.stationary.Nash}.
    Then, $x^*\in X$ is a stationary Nash equilibrium for {\NEPeps}.
\end{proposition}
\begin{proof}
    Let $(x^\nu,x^{*,-\nu})=(x^{*,1},\dots,x^{*,\nu-1},x^\nu,x^{*,\nu+1},\dots,x^{*,N})$, where $x^\nu\in X^\nu$ is arbitrary, \eqref{ieq:VI.stationary.Nash} is reduced to
    \begin{align}\label{ieq:stationary.Nash.i}
        \langle\nabla_{x^\nu}\Theta^\nu_\varepsilon(x^{*,\nu},x^{*,-\nu}),x^\nu-x^{*,\nu}\rangle\ge0\quad\forall x^\nu\in X^\nu.
    \end{align}
    By Lemma~\ref{lem:regular.sufficient}, \eqref{ieq:stationary.Nash.i} is equivalent to
    $$
    \begin{aligned}
        &0\in\nabla_{x^\nu}\Theta^\nu(x^{*,\nu},x^{*,-\nu};\varepsilon)+{\cal N}_{X^\nu}(x^{*,\nu})\\
        \iff&0\in\nabla_{x^\nu}\Theta^\nu(x^{*,\nu},x^{*,-\nu};\varepsilon)+{\cal T}_{X^\nu}(x^{*,\nu})^\circ,
    \end{aligned}
    $$
    which coincides with \eqref{eq:stationary.NE}.
    The claim holds for all $\nu$, and thus $x^*\in X$ is a stationary Nash equilibrium for {\NEPeps}.
\end{proof}

\begin{proposition}\label{prop:existence.stationary.Nash}
    Suppose that Assumptions~\ref{asmp:LF} and~\ref{asmp:followersVI} hold, and $X^\nu\subset\Re^{n_\nu}$ is convex for all $\nu$.
    For each $\varepsilon>0$, there exists a stationary Nash equilibrium to {\NEPeps}.
\end{proposition}
\begin{proof}
    It suffices to show that there exists a solution to variational inequality \eqref{ieq:VI.stationary.Nash} by the convexity of $X$.
    Since $F^\ell_\varepsilon$ is continuous, the solution set of \eqref{ieq:VI.stationary.Nash} is nonempty and compact by \cite[Corollary 2.2.5]{Facchinei2003}.
\end{proof}

\subsection{Convergence to stationary Nash equilibrium}

%In what follows, the strategy set $X^\nu$ of leader $\nu$ is defined by
%$$
%    X^\nu:=\{x^\nu\in\Re^{n_\nu}\mid h^\nu(x^\nu)\le0\},
%$$
%where $h^\nu\colon\Re^n\to\Re^{p_i}$ is continuously differentiable.
%Let $p:=p_1+\dots+p_N$.

%We assume the Mangasarian--Fromovitz constraint qualification for all leaders' optimization problems.
%
%\begin{assumption}\label{asmp:MFCQ.leader}
%    For each $\nu$, the Mangasarian--Fromovitz constraint qualification (MFCQ) holds; that is, for $x^\nu\in X^\nu$, there exists a vector $d\in\Re^{n_\nu}$ such that
%    $$
%        \nabla_{x^\nu}h^\nu_\omega(x^\nu)^\top d<0,\quad j:h^\nu_\omega(x^\nu)=0\ (j=1,\dots,p_i)
%    $$
%\end{assumption}

In this subsection we show that the sequence of the stationary Nash equilibrium of {\NEPepsk} converges to the B-/C-stationary Nash equilibrium of {\NEP} for appropriately chosen $\{\varepsilon_k\}$.

\begin{theorem}\label{thm:stationarity}
    Suppose that Assumptions~\ref{asmp:LF}, \ref{asmp:followersVI}, and \eqref{eq:inclusion.i} hold for all $\nu\in\{1,\dots,N\}$.
    Assume that $X^\nu$ is regular for all $\nu$.
    Let $\{x^k\}$ be a sequence of stationary Nash equilibria of {\NEPepsk}, i.e., $x^{k,\nu}$ satisfies \eqref{eq:stationary.NE} for all $\nu$.
    Then every accumulation point $x^*$ of the sequence $\{x^k\}$ is a C-stationary Nash equilibrium of {\NEP}.

    Moreover, assume that the reduced cost function $\Theta^\nu(x^\nu,x^{-\nu}):=\theta^\nu(x^\nu,x^{-\nu},y(x^\nu,x^{-\nu}))$ is regular with respect to $x^\nu$ at $x^*$ for all $\nu$.
    Then, $x^*$ is a B-stationary Nash equilibrium of {\NEP}.
\end{theorem}
\begin{proof}
    Since $x^{k,\nu}$ is a stationary Nash equilibrium of {\NEPepsk}, for leader $\nu$ $x^{k,\nu}\in X^\nu$, under the regularity of $X^\nu$, satisfies
    \begin{align*}
        &0\in\nabla_{x^\nu}\Theta^\nu_{\varepsilon_k}(x^{k,\nu},x^{k,-\nu})+{\cal N}_{X^\nu}(x^{k,\nu})\nonumber\\
        \iff&0\in\nabla_{x^\nu}\theta^\nu(x^{k,\nu},x^{k,-\nu},y_{\varepsilon_k}(x^{k,\nu},x^{k,-\nu}))+\nonumber\\
        &\qquad\nabla_{x^\nu}y_{\varepsilon_k}(x^{k,\nu},x^{k,-\nu})\nabla_y\theta^\nu(x^{k,\nu},x^{k,-\nu},y_{\varepsilon_k}(x^{k,\nu},x^{k,-\nu}))+{\cal N}_{X^\nu}(x^{k,\nu})%\label{eq:sta.Nash.k}
    \end{align*}
    By the compactness of $X$, we can assume that $x^*$ is an accumulation point of $\{x^k\}$ without loss of generality.
    The continuity of $\nabla_{x^\nu}\theta^\nu$, $\nabla_y\theta^\nu$, and $y(\cdot)$ along with Proposition~\ref{prop:implicit.y.z.lmd} implies that
    $$
        \begin{aligned}
            \lim_{k\to\infty}\nabla_{x^\nu}\theta^\nu(x^{k,\nu},x^{k,-\nu},y_{\varepsilon_k}(x^{k,\nu},x^{k,-\nu}))=\nabla_{x^\nu}\theta^\nu(x^{*,\nu},x^{*,-\nu},y(x^{*,\nu},x^{*,-\nu})),\\
            \lim_{k\to\infty}\nabla_{y}\theta^\nu(x^{k,\nu},x^{k,-\nu},y_{\varepsilon_k}(x^{k,\nu},x^{k,-\nu}))=\nabla_{y}\theta^\nu(x^{*,\nu},x^{*,-\nu},y(x^{*,\nu},x^{*,-\nu})).
            %\lim_{k\to\infty}{\cal N}_{X^\nu}(x^{k,\nu})={\cal N}_{X^\nu}(x^{*,\nu}).
        \end{aligned}
    $$
    From the assumption of \eqref{eq:inclusion.i}, there exists a matrix $V^{*,\nu}\in\Re^{n_\nu\times m}$ such that
    \begin{align}\label{eq:matrix.subdifferential}
        V^{*,\nu}=\lim_{k\to\infty}\nabla_{x^\nu}y_{\varepsilon_k}(x^{k,\nu},x^{k,-\nu})\in\partial_{x^\nu}y(x^{*,\nu},x^{*,-\nu}).
    \end{align}
    Since $x^k\to x^*$, $-\nabla_{x^\nu}\Theta^\nu_{\varepsilon_k}(x^{k,\nu},x^{k,-\nu})\in{\cal N}_{X^\nu}(x^{k,\nu})$, and $\nabla_{x^\nu}\Theta^\nu_{\varepsilon_k}(x^{k,\nu},x^{k,-\nu})\to\nabla_{x^\nu}\Theta^\nu(x^{*,\nu},x^{*,-\nu})$ from the observation above, by \cite[Proposition~6.6]{Rockafellar1998}, we have $-\nabla_{x^\nu}\Theta^\nu(x^{*,\nu},x^{*,-\nu})\in{\cal N}_{X^\nu}(x^{*,\nu})$, which implies that
    \begin{align}
        0\in&\nabla_{x^\nu}\theta^\nu(x^{*,\nu},x^{*,-\nu},y(x^{*,\nu},x^{*,-\nu}))\nonumber\\
        &\qquad+V^{*,\nu}\nabla_y\theta^\nu(x^{*,\nu},x^{*,-\nu},y(x^{*,\nu},x^{*,-\nu}))+{\cal N}_{X^\nu}(x^{*,\nu}).\label{ge:limit.Nash}
    \end{align}
    Since $\theta^\nu$ is strictly differentiable with respect to $y$, the Jacobian chain rule \cite[Theorem 2.6.6]{Clarke1990} can be applied and then yields
    $$
        \begin{aligned}
        \partial_{x^\nu}\Theta^\nu(x^{*,\nu},x^{*,-\nu})=\nabla_{x^\nu}&\theta^\nu(x^{*,\nu},x^{*,-\nu},y(x^{*,\nu},x^{*,-\nu}))\\
        &+\partial_{x^\nu}y(x^{*,\nu},x^{*,-\nu})\nabla_y\theta^\nu(x^{*,\nu},x^{*,-\nu},y(x^{*,\nu},x^{*,-\nu})).
        \end{aligned}
    $$
    Then \eqref{ge:limit.Nash} implies $0\in\partial_{x^\nu}\Theta^\nu(x^{*,\nu},x^{*,-\nu})+{\cal N}_{X^\nu}(x^{*,\nu})$.
    The claim simultaneously holds for all $\nu$; thus, $x^*$ is a C-stationary Nash equilibrium of {\NEP}.
    
    Now we show the convergence to B-stationary Nash equilibrium by assuming the regularity of $\Theta^\nu$.
    Since $X^\nu$ is regular for all $\nu$, $X$ is also regular.
    By \cite[Proposition 4.6.3]{Bertsekas2003}, \eqref{ge:limit.Nash} is equivalent to
    $$
    \begin{aligned}
        \nabla_{x^\nu}&\theta^\nu(x^{*,\nu},x^{*,-\nu},y(x^{*,\nu},x^{*,-\nu}))^\top d^\nu+\\
        &\qquad{d^\nu}^\top V^{*,\nu}\nabla_y\theta^\nu(x^{*,\nu},x^{*,-\nu},y(x^{*,\nu},x^{*,-\nu}))\ge0\quad\forall d^\nu\in{\cal T}_{X^\nu}(x^{*,\nu}).
    \end{aligned}
    $$
    The regularity assumption of $\Theta^\nu$ and \cite[Proposition 2.1.2 (b)]{Clarke1990} lead that
    $$
    \begin{aligned}
        (\Theta^\nu)'&(x^{*,\nu},x^{*,-\nu};d^\nu)=\max\{{\zeta^\nu}^\top d^\nu\mid\zeta^\nu\in\partial_{x^\nu}\Theta^\nu(x^{*,\nu},x^{*,-\nu})\}\ge\\
        &\nabla_{x^\nu}\theta^\nu(x^{*,\nu},x^{*,-\nu},y(x^{*,\nu},x^{*,-\nu}))^\top d^\nu\\
        &\quad\qquad+{d^\nu}^\top V^{*,\nu}\nabla_y\theta^\nu(x^{*,\nu},x^{*,-\nu},y(x^{*,\nu},x^{*,-\nu}))\ge0\quad\forall d^\nu\in{\cal T}_{X^\nu}(x^{*,\nu}).
    \end{aligned}
    $$
    The above assertion holds for every $\nu$, which implies that $x^*$ is a B-stationary Nash equilibrium point of {\NEP}.
\end{proof}

\begin{remark}
    Hori and Fukushima \cite{Hori2019} showed the convergence to B-stationary Nash equilibrium with a squared penalty method for an EPEC associated with MLMFG, which means that the case when the penalty parameter $\rho_k\to\infty$.
    As is well known, the squared penalty method is easily failed to be ill-conditioned.
    Then they proposed a refinement procedure after obtaining an `approximated' B-stationary Nash equilibrium with the penalty method.
    However, the accumulation point obtained with the refinement may not guarantee the B-stationarity but only the weak stationarity because both problems are intrinsically different from each other.
    Meanwhile, our method enables us to obtain the B-stationary Nash equilibrium accurately.
\end{remark}

If $X^\nu\subset\Re^{n_\nu}$ is given by
\begin{align}\label{X.ieq.eq.system}
    X^\nu:=\{x^\nu\in\Re^{n_\nu}\mid u^\nu(x^\nu)\le0,\ v^\nu(x^\nu)=0\},
\end{align}
where $u^\nu\colon\Re^{n_\nu}\to\Re^{p_\nu}$ and $v^\nu\colon\Re^{n_\nu}\to\Re^{q_\nu}$ are continuously differentiable, the convergence result of Theorem~\ref{thm:stationarity} can be shown without regularity assumption of $X^\nu$ under an appropriate constraint qualification.

\begin{definition}\label{def:MFCQ}
    For each $\nu\in\{1,\dots,N\}$, we say that the \emph{Mangasarian--Fromovitz constraint qualification (MFCQ)} holds at $\bar{x}^\nu\in X^\nu$ if $\nabla v^\nu_j(\bar{x}^\nu)$, $j=1,\dots,q_\nu$, are linearly independent, and there exists $d^\nu\in\Re^{n_\nu}$ such that $\langle\nabla u^\nu_i(\bar{x}^\nu),d^\nu\rangle<0$ for all $i\in{\cal I}(\bar{x}^\nu):=\{i\mid u^\nu_i(\bar{x}^\nu)=0\}$ and $\langle\nabla v^\nu_j(\bar{x}^\nu),d\rangle=0$ for all $j=1,\dots,q_\nu$.
\end{definition}

\begin{corollary}\label{thm:conv.nonconvex}
    Suppose that Assumptions~\ref{asmp:LF} and~\ref{asmp:followersVI} hold, and assume that \eqref{eq:inclusion.i} holds for all $\nu\in\{1,\dots,N\}$.
    Suppose also that the MFCQ holds for all $\nu$ at every accumulation point $x^*$ of the sequence $\{x^k\}$.
    Then $x^*$ is a C-stationary Nash equilibrium for {\NEP}.

    Moreover, if $\Theta^\nu$ is regular with respect to $x^\nu$ at $x^*$ for all $\nu$, then $x^*$ is a B-stationary Nash equilibrium for {\NEP}.
\end{corollary}
\begin{proof}
    Under the MFCQ assumption, there exists Lagrange multipliers $\zeta^{\le,k,\nu}\in\Re^{p_\nu}_+$ and $\zeta^{=,k,\nu}\in\Re^{q_\nu}$ satisfying the following optimality condition of \eqref{prob:approximated.reduced.leader.i}:
    \begin{align}\label{eq:KKT.nu.eps}
        \begin{aligned}
        \nabla_{x^\nu}\Theta^\nu_{\varepsilon_k}(x^{k,\nu},x^{k,-\nu})+\nabla u^\nu(x^{k,\nu})\zeta^{\le,k,\nu}+\nabla v^\nu(x^{k,\nu})\zeta^{=,k,\nu}=0,\\
        0\le\zeta^{\le,\nu}\perp -u^\nu(x^{k,\nu})\ge 0,\\
        v^\nu(x^{k,\nu})=0.
        \end{aligned}
    \end{align}
    By the continuity of $\nabla_{x^\nu}\theta^\nu$, $\nabla_y\theta^\nu$, and $y$ along with Proposition~\ref{prop:implicit.y.z.lmd}, there exists a limit $(x^*,\zeta^*)$ of KKT conditions \eqref{eq:KKT.nu.eps} because the sequence $\{(x^k,\zeta^k)\}$ of KKT tuple is bounded by the MFCQ.
    Then we obtain 
    \begin{gather}
        \nabla_{x^\nu}\theta^\nu(x^{*,\nu},x^{*,-\nu},y(x^{*,\nu},x^{*,-\nu}))+V^{*,\nu}\nabla_y\theta^\nu(x^{*,\nu},x^{*,-\nu},y(x^{*,\nu},x^{*,-\nu}))\nonumber\\
        \qquad+\nabla u^\nu(x^{*,\nu})\zeta^{\le,*,\nu}+\nabla v^\nu(x^{*,\nu})\zeta^{=,*,\nu}=0,\label{eq:KKT.limit.MFCQ1}\\
        0\le\zeta^{\le,*,\nu}\perp -u^\nu(x^{*,\nu})\ge 0,\label{eq:KKT.limit.MFCQ2}\\
        v^\nu(x^{*,\nu})=0,\label{eq:KKT.limit.MFCQ3}
    \end{gather}
    where $V^{*,\nu}$ is defined in~\eqref{eq:matrix.subdifferential}.
    It follows from the Jacobian chain rule that \eqref{eq:KKT.limit.MFCQ1} is
    \begin{align}\label{eq:KKT.limit.MFCQ1.inclusion}
        0\in&\nabla_{x^\nu}\theta^\nu(x^{*,\nu},x^{*,-\nu},y(x^{*,\nu},x^{*,-\nu}))\nonumber\\
        &\quad+\partial_{x^\nu}y(x^{*,\nu},x^{*,-\nu})\nabla_y\theta^\nu(x^{*,\nu},x^{*,-\nu},y(x^{*,\nu},x^{*,-\nu}))\nonumber\\
        &\quad+\nabla u^\nu(x^{*,\nu})\zeta^{\le,*,\nu}+\nabla v^\nu(x^{*,\nu})\zeta^{=,*,\nu}.
    \end{align}
    Under the MFCQ, KKT conditions \eqref{eq:KKT.limit.MFCQ1.inclusion}, \eqref{eq:KKT.limit.MFCQ2}, and \eqref{eq:KKT.limit.MFCQ3} coincide with $0\in\partial_{x^\nu}\Theta^\nu(x^{*,\nu},x^{*,-\nu})+{\cal T}_{X^\nu}(x^{*,\nu})^\circ$, and thus $x^*$ is a C-stationary Nash equilibrium of {\NEP}.

    The latter claim can also be shown by the same manner as Theorem~\ref{thm:stationarity}.
\end{proof}

\begin{remark}
    It is easy to see that B- and C-stationarity are equivalent when $z$ and $\lambda$ satisfy the strict complementarity; that is, $z_i+\lambda_i>0$ for all $i=1,\dots,l$.
\end{remark}
%\subsection{Special case of MLMFG}
%
%We discussed finding a stationary Nash equilibrium for $\mathrm{NEP}_0(X,\{\Theta^\nu\})$.
%In some special cases, we can claim the existence of a L/F Nash equilibrium for the MLMFG.
%To this end, suppose that the following assumptions hold in addition to Assumption~\ref{asmp:LF}.
%\begin{assumption}\label{asmp:LF.Nash}
%\begin{enumerate}
%    \item For all $\nu$, $X^\nu$ is convex;
%    \item $\theta^\nu(\cdot,x^{-\nu},\cdot)$ is jointly convex for arbitrary fixed $x^{-\nu}$.
%\end{enumerate}
%\end{assumption}
%
%Consider that follower $j\in\{1,\dots,M\}$ solves a quadratic programming as follows:
%$$
%    \begin{aligned}
%    \min_{y^\omega\in\Re^{m_\omega}}&\quad\frac{1}{2}\langle y^\omega,Q_\omegay^\omega\rangle+\langle q^\omega,y^\omega\rangle \\
%    \text{s.t.}&\quad A_\omegay^\omega+B_\omegax+c^\omega\le0,\\
%    &\quad y^\omega\ge0,
%    \end{aligned}
%$$
%where $Q_\omega\in\Re^{m_\omega\times m_\omega}$ is positive definite, $q^\omega\in\Re^{m_\omega}$, $A_\omega\in\Re^{l_\omega\times m_\omega}$, 

%\begin{remark}
%    Herty et al. \cite{Herty2022} considered the multi-leader--single-follower game in which the follower's response can be written by the following closed form:
%    $$
%        y(x)=\max\{Q^{-1}b(x),l(x)\}
%    $$
%    
%\end{remark}

%\red{
%\begin{remark}
%    When $X^\nu$ is given by~\eqref{X.ieq.eq.system}, and the equality and inequality systems satisfy the MFCQ, the transition from the C-stationarity to the B-stationarity needs the regularity of $\Theta^\nu(\cdot,x^{*,-\nu})$ for all $\nu$
%\end{remark}
%}

\section{Numerical experiments}\label{sec:exp}
In this section, we report results of numerical experiments conducted to illustrate the behavior of the proposed method with a toy example.
First, we introduce a two-leader--two-follower game, i.e., $N=2$ and $M=2$, to check the validity of the smoothing method.

We refer to the extended model of Hori and Fukushima \cite{Hori2019} with the nonnegative constraint $x^\nu\ge0$ on each leader's optimization problem. %\footnote{Although we consider the case of one follower, the method is the same as we presented in previous sections, which implies that for this numerical example, the number of followers is not essential.}.
Leader $\nu\in\{1,2\}$ solves the following problem:
\begin{align}\label{prob.ex:leader}
    \begin{aligned}
        \min_{x^\nu\in\Re^2}\quad&\frac{1}{2}(x^\nu)^\top H_\nu x^\nu+(x^\nu)^\top G_{\nu,-\nu}x^{-\nu}+\sum_{\omega=1,2} (x^\nu)^\top D_{\nu,\omega} y^\omega + (q^\nu)^\top x^\nu\\
        \text{s.t.}\quad& A_\nu x^\nu\le b^\nu,\ x^\nu\ge0,
    \end{aligned}
\end{align}
where $H_\nu\in\Re^{n_\nu\times n_\nu}$, $G_{\nu,-\nu}\in\Re^{n_\nu\times n_{-\nu}}$, $D_\nu\in\Re^{n_\nu\times m}$, $A_\nu\in\Re^{p_\nu\times n_\nu}$, and $b^\nu\in\Re^{p_\nu}$.
Meanwhile, follower $\omega\in\{1,2\}$ solves the following problem:
\begin{align}\label{prob.ex:follower}
    \begin{aligned}
        \min_{y^\omega\in\Re^2}\quad&\frac{1}{2}(y^\omega)^\top M_\omega y^\omega+(y^\omega)^\top Q_{\omega,-\omega} y^{-\omega} -\sum_{\nu=1,2}(x^\nu)^\top D_{\nu,\omega} y^\omega\\
        \text{s.t.}\quad& (c^\omega)^\top y^\omega+\sum_{\nu=1,2} (d^\nu)^\top x^\nu+a_\omega\ge0,\ 
        y^\omega\ge 0.
    \end{aligned}
\end{align}
where $M_\omega$, $\omega=1,2$, are symmetric.
Although it is a quadratic game, it cannot be solved by the approaches given in the previous works~\cite{Hu2011} and \cite{Herty2022} because the response $y(x)$ cannot be obtained explicitly due to the non-diagonal matrices $M_\omega$ and the coefficient vector $c^\omega$ of $y^\omega$ as also indicated in Remark~\ref{rem:assumption}.

The KKT conditions of \eqref{prob.ex:follower}, $\omega=1,2$, are given as follows:
$$
    \begin{aligned}
    &0\le\left[\begin{array}{c} y^1 \\ y^2 \end{array}\right] \perp
    \left[\begin{array}{cc}
        M_1 & Q_{1,2} \\
        Q_{2,1} & M_2
    \end{array}\right]
    \left[\begin{array}{c} y_1 \\ y_2 \end{array}\right]
    -\left[\begin{array}{cc}
        D_{1,1}^\top & D_{2,1}^\top \\
        D_{1,2}^\top & D_{2,2}^\top
    \end{array}\right]
    \left[\begin{array}{c} x_1 \\ x_2 \end{array}\right]
    -\left[\begin{array}{cc}
        c^1 & 0 \\
        0 & c^2
    \end{array}\right]
    \left[\begin{array}{c} \lambda_1 \\ \lambda_2 \end{array}\right]\ge 0,\\
    &0\le\left[\begin{array}{c} \lambda_1 \\ \lambda_2 \end{array}\right] \perp
    \left[\begin{array}{cc}
        (c^1)^\top & 0 \\
        0 & (c^2)^\top
        \end{array}\right]\left[\begin{array}{c}y^1 \\ y^2\end{array}\right]
     +\left[\begin{array}{cc}
       (d^1)^\top & (d^2)^\top \\
       (d^1)^\top & (d^2)^\top
       \end{array}\right]\left[\begin{array}{c}x^1 \\ x^2\end{array}\right]
     +\left[\begin{array}{c} a_1\\ a_2 \end{array}\right]\ge 0.
    \end{aligned}
$$
To ensure the uniqueness of $y$, we assume that 
$$
    \left[
    \begin{array}{cc}
        M_1 & Q_{1,2} \\
        Q_{2,1} & M_2
    \end{array}
    \right]
$$
is positive definite.
Let $y_\varepsilon(x)$ be a solution to the perturbed KKT conditions of \eqref{prob.ex:follower} with the smoothing method.
Here, we define
\[
    A:=\left[\begin{array}{cc}A_1 & \\ & A_2\end{array}\right],\ b:=\left[\begin{array}{c} b^1 \\ b^2 \end{array}\right],\ X:=\{x\in\Re^n\mid Ax\le b,\ x\ge0 \},
\]
and
\[
    F^\ell_\varepsilon(x):=\left[\begin{array}{c}\nabla_{x^1}\Theta^1_\varepsilon(x^1,x^2) \\ \nabla_{x^2}\Theta^2_\varepsilon(x^1,x^2)\end{array}\right],
\]
where
\[
\begin{aligned}
    \nabla_{x^\nu}\Theta^\nu_\varepsilon(x^\nu,x^{-\nu})=&H_\nu x^\nu+G_{\nu,-\nu}x^{-\nu}+\\
    &\sum_{\omega=1,2}\Bigl(D_{\nu,\omega} y^\omega_\varepsilon(x^\nu,x^{-\nu})+\nabla_{x^\nu}y^\omega_\varepsilon(x^\nu,x^{-\nu})D_{\nu,\omega}^\top x^\nu\Bigr).
\end{aligned}
\]
Then the stationary Nash equilibrium of {\NEPeps} must satisfy the following VI: Find $x^*_\varepsilon\in X$ such that
\begin{align}\label{vi.ex}
    \langle F^\ell_\varepsilon(x^*_\varepsilon),x-x^*_\varepsilon\rangle\ge0\quad\forall x\in X,
\end{align}
and its KKT conditions are written as the nonlinear complementarity problem:
\begin{align}
    0\le F^\ell_\varepsilon(x)+A^\top\mu&\perp x\ge 0,\label{ex.NCP1}\\
    0\le b-Ax&\perp\mu\ge0.\label{ex.NCP2}
\end{align}
Let $v:=(x,\mu)\in\Re^{n+p}$ and
\[
    \hat{F}^\varepsilon(v):=\left[\begin{array}{c}F^\ell_\varepsilon(x)+A^\top\mu \\ b-Ax \end{array}\right],\ 
    \Psi^\varepsilon(v):=\left[\begin{array}{c}\phi_0(v_1,\hat{F}^\varepsilon_1(v))\\\vdots\\\phi_0(v_{n+p},\hat{F}^\varepsilon_{n+p}(v))\end{array}\right].
\]
Then NCP \eqref{ex.NCP1}--\eqref{ex.NCP2} is reformulated as the nonsmooth equation $\Psi^\varepsilon(v)=0$.
We solve NCP \eqref{ex.NCP1}--\eqref{ex.NCP2} by semismooth Newton's method proposed for NCP by Luca et al. \cite{Luca1996}, and we use $1/2\|\Psi^\varepsilon(v)\|^2_2$ as a merit function in the line search algorithm.
The stopping criterion is given by
\[
    \|\min\{v,\hat{F}^\varepsilon(v)\}\|_\infty<(n+p)\times10^{-6},
\]
where the $\min$ operator means the componentwise minimum value.
The numerical instances are shown below:
\begin{gather*}
    n_1=n_2=2,\ p_1=p_2=2,\ m_1=m_2=2,\\
    H_1=\left[\begin{array}{rr}3&-4\\-4&2\end{array}\right],
    H_2=\left[\begin{array}{rr}4&-5\\-5&-3\end{array}\right],
    G_{1,2}=\left[\begin{array}{rr}2&-1\\2&2\end{array}\right],
    G_{2,1}=-G_{1,2}^\top,\\%\left[\begin{array}{rr}2&-1\\-1&2\end{array}\right],\\
    D_{1,1}=\left[\begin{array}{rr}1&2\\2&1\end{array}\right],
    D_{2,1}=\left[\begin{array}{rr}2&1\\1&1\end{array}\right],
    D_{1,2}=\left[\begin{array}{rr}1&2\\1&1\end{array}\right],
    D_{2,2}=\left[\begin{array}{rr}2&1\\1&2\end{array}\right],\\
    q^1=\left[\begin{array}{c}-6\\-6\end{array}\right],
    q^2=\left[\begin{array}{c}-6\\-6\end{array}\right],
    A_1=\left[\begin{array}{rr}2&1\\1&2\end{array}\right],
    A_2=\left[\begin{array}{rr}1&2\\2&1\end{array}\right],
    b^1=\left[\begin{array}{c}3\\1\end{array}\right],
    b^2=\left[\begin{array}{c}3\\1\end{array}\right],\\
    M_1=\left[\begin{array}{rr}3&1\\1&3\end{array}\right],
    M_2=\left[\begin{array}{rr}2&1\\1&3\end{array}\right],
    Q_{1,2}=\left[\begin{array}{rr}1&1\\1&2\end{array}\right],
    Q_{2,1}=-Q_{1,2}^\top,\\
    c^1=\left[\begin{array}{c}-1\\-1\end{array}\right],
    c^2=\left[\begin{array}{c}-1\\-1\end{array}\right],
    d^1=\left[\begin{array}{c}1\\1\end{array}\right],
    d^2=\left[\begin{array}{c}1\\1\end{array}\right],
    a_1=4,
    a_2=4.
\end{gather*}
%\begin{gather*}
%    n_1=n_2=2,\ p_1=p_2=2,\ m=2,\\
%    H_1=\left[\begin{array}{rr}8&-5\\-5&1\end{array}\right],
%    H_2=\left[\begin{array}{rr}2&-5\\-5&5\end{array}\right],
%    G_{1,2}=\left[\begin{array}{rr}1&-1\\-1&1\end{array}\right],
%    G_{2,1}=\left[\begin{array}{rr}2&-1\\-1&2\end{array}\right],\\
%    D_1=\left[\begin{array}{rr}1&-1\\-2&1\end{array}\right],
%    D_2=\left[\begin{array}{rr}2&2\\-1&-2\end{array}\right],
%    A_1=\left[\begin{array}{rr}2&1\\1&1\end{array}\right],
%    A_2=\left[\begin{array}{rr}1&1\\1&2\end{array}\right],\\
%    M=\left[\begin{array}{rr}2&1\\1&1\end{array}\right],
%    q=\left[\begin{array}{c}-1\\-1\end{array}\right],
%    c=\left[\begin{array}{c}1\\1\end{array}\right],
%    d_1=\left[\begin{array}{c}-1\\-1\end{array}\right],
%    d_2=\left[\begin{array}{c}-2\\-2\end{array}\right],
%    a=5.
%\end{gather*}

We solve NCP \eqref{ex.NCP1}--\eqref{ex.NCP2} sequentially as $\varepsilon_k$ decreases, where $\varepsilon_k=0.9^{k}$, $k=0,\dots,74$.
We run the algorithm with the initial point $x^0:=(3,3,3,3)$.
Here, $y,z,\lambda$ is unique then without initial points, they are uniquely determined depending on $x^0$.
Figures~\ref{fig:ex1.l1}--~\ref{fig:ex1.f2} depict the stationary Nash equilibrium for $\varepsilon_k$, $k=0,\dots,74$, and the sequence $\{x^k\}$ converges to $x^*$.
We set other initial points, but the convergent point and the curve are similar to those figures.
\begin{figure}[ht]
  \centering
  \begin{minipage}[b]{0.49\textwidth}
    \includegraphics[width=\textwidth]{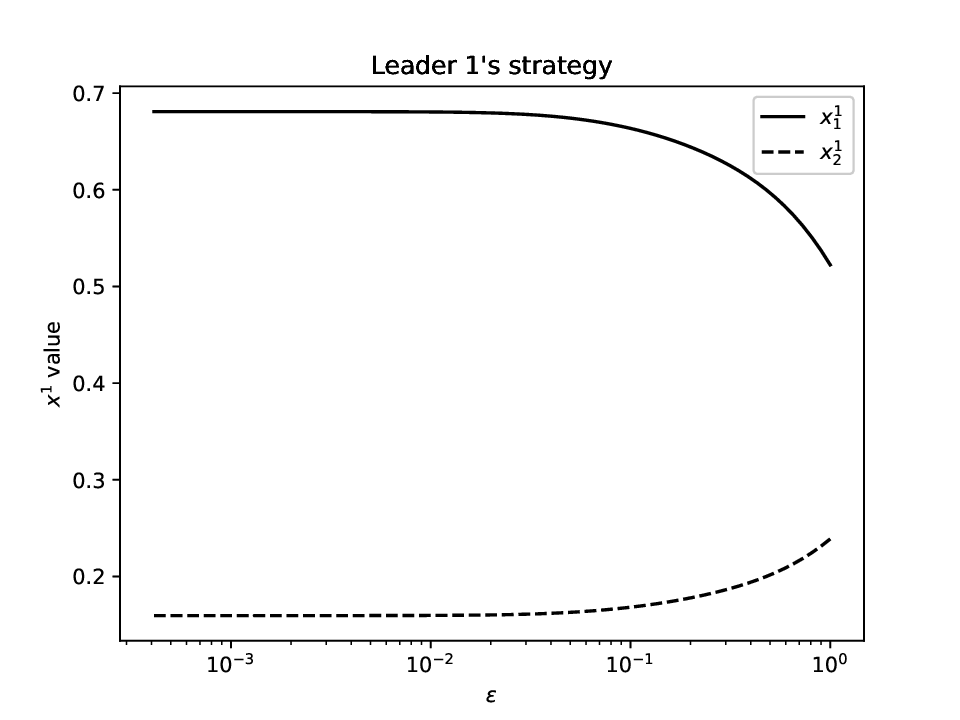}
    \subcaption{Leader 1's strategy $x^1_{\varepsilon_k}$}\label{fig:ex1.l1}
  \end{minipage}
  \hfill
  \begin{minipage}[b]{0.49\textwidth}
    \includegraphics[width=\textwidth]{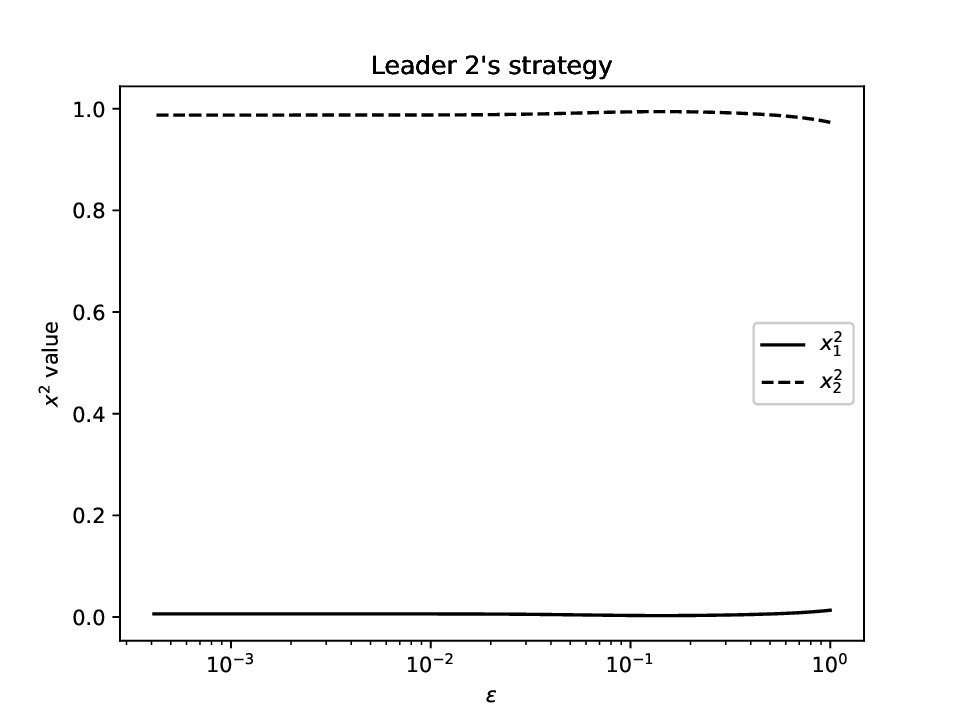}
    \subcaption{Leader 2's strategy $x^2_{\varepsilon_k}$}\label{fig:ex1.l2}
  \end{minipage}
  \hfill
  \begin{minipage}[b]{0.49\textwidth}
    \includegraphics[width=\textwidth]{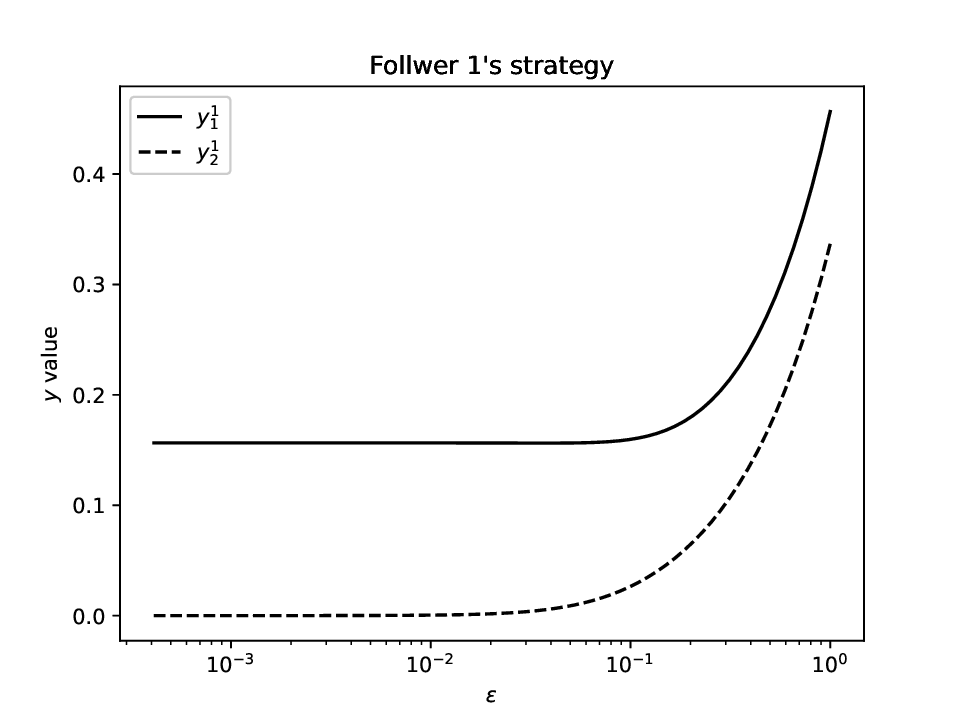}
    \subcaption{Follower's strategy $y^1_{\varepsilon_k}(x_{\varepsilon_k})$}\label{fig:ex1.f1}
  \end{minipage}
  \hfill
  \begin{minipage}[b]{0.49\textwidth}
    \includegraphics[width=\textwidth]{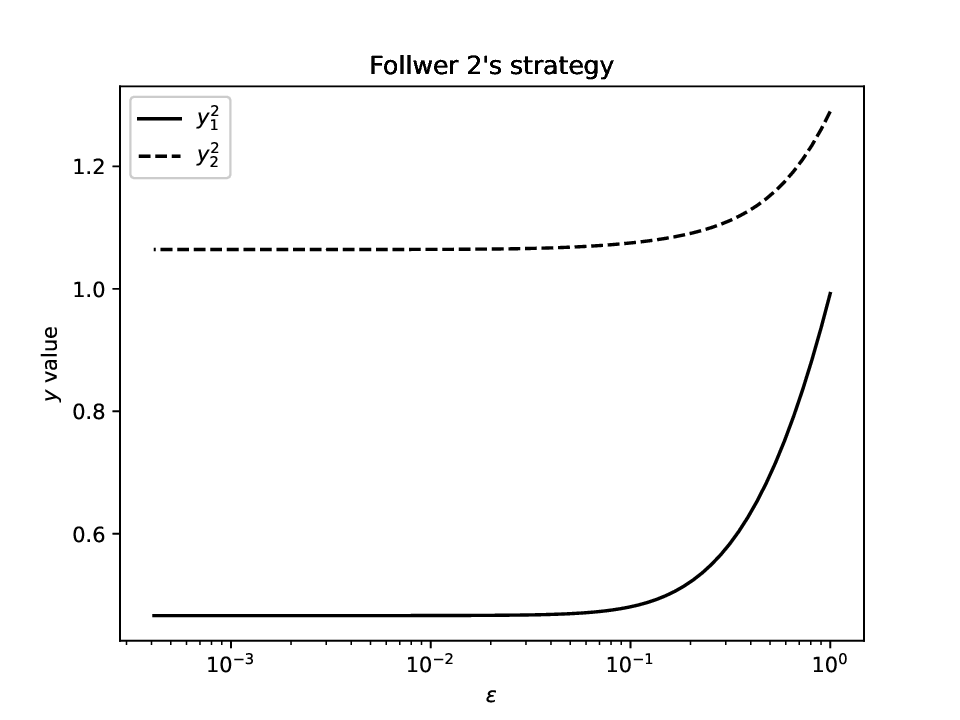}
    \subcaption{Follower's strategy $y^2_{\varepsilon_k}(x_{\varepsilon_k})$}\label{fig:ex1.f2}
  \end{minipage}
  \caption{Sequence of the stationary Nash equilibrium.}
\end{figure}
\begin{remark}
    The semismooth Newton's method uses $\nabla F^\ell_\varepsilon(x)$ with the Hessian matrix $\nabla^2 y_\varepsilon(x)$ of the response function $y_\varepsilon(x)$, which requires the Hessian of $H_\varepsilon(x,y,z,\lambda)$ with respect to $(y,z,\lambda)$ and so on.
    However, the formula of $\nabla^2 y_\varepsilon(x)$ is much complicated to obtain, and thus we use a numerical differentiation.
    Hence, the accuracy of the numerical results when $\varepsilon$ is very small may not be guaranteed. 
    %Furthermore, we observe that for $\varepsilon\le10^{-3}$, the stationary Nash equilibrium is almost close to the case where $\varepsilon=0$; however, as the numerical instance is large in the followers' problems, we may have to start larger $\varepsilon$.
\end{remark}
We now conclude this section.
We verified the behavior of the proposed method with toy examples, and all the solutions $x^k$ and $y_{\varepsilon_k}(x^k)$ converge to the B-stationary Nash equilibrium as $\varepsilon_k\to0$.
Meanwhile, we require the Jacobian matrix of $F^\ell_\varepsilon(x)$ to solve variational inequality \eqref{vi.ex} with gradient-based method which includes the Hessian matrix of the response function $y_\varepsilon(x)$.
Hence the computation of the Hessian matrix and its inverse is, of course, expensive when the dimension of followers' problems is high.
We leave this issue as a topic for future research.
%\subsection{Application to edge computing}
%
%With the technology development of Internet of Things (IoT) and mobile Internet, there is an increasing demand in computing resources, which brings more pressure in the cloud centers, and higher latency requirement in the communication networks \cite{Lyu2022}.
%To overcome such an excess concentration of computation on cloud centers, multi-access edge computing (MEC) is proposed recently by \cite{Satyanarayanan2017}.
%Edge computing is widely used in, for example, autonomous driving, agritech, and smart factory.
%The advantage of edge computing is not only reducing the concentration but also reducing the cost of communication and protecting users' privacy.
%Optimal computational resource allocation between edge nodes and user devices is important to achieve an efficient computation.
%In more recent years, a game theoretic approach to optimal resource allocation is attracting attention \cite{Lyu2022,xxx,yyy}.
%
%Following \cite{Xiong2019}, consider a public blockchain supported by cloud/edge computing under the proof-of-concept (PoW)-based protocol.
%In the blockchain network, there are groups of miners and providers.
%Since miners have limited computing power, they want to access and consume the computing service from a group of providers.

\section{Conclusion}\label{sec:conclusion}
This paper has presented a smoothing method of the followers' best response in the MLMFG in the case where the followers' optimization problems are more general than the previous studies \cite{Herty2022,Hu2011}.
Then we have shown the convergence result for a Clarke and Bouligand stationary Nash equilibrium on MLMFG as the smoothing parameter tends to zero.
Finally, a numerical experiment has been conducted to confirm the validity of the proposed method; specifically, observing the behavior of the sequence of stationary Nash equilibria as the smoothing parameter approaches zero.

%One may concern the numerical evaluation of $F^\ell_\varepsilon(x)$.
%In fact, solving variational inequality \eqref{ieq:VI.stationary.Nash} is demanding as we showed in Section~\ref{sec:NumericalExp}.
%During the line search in Algorithm~\ref{alg:semismooth}, this algorithm has to evaluate the merit function for each step size.
%In the method, the most computational burden is to obtain $\nabla y_\varepsilon(x)$ since it includes the inverse matrix of $\nabla_{y,z,\lambda}H_\varepsilon(x,y,z,\lambda)$.
%Efficient computation of $\nabla y_\varepsilon(x)$ is the future study.
%
Remark that in this paper we did not mention a local surrogate of the LF Nash equilibrium for MLMFG.
To the best of the authors' knowledge, there is no such literature on MLMFG that deals with a local LF Nash equilibrium.
Unfortunately, verifying local optimality in nonconvex optimization is still NP-hard, and thus as well as finding an LF Nash equilibrium in MLMFG, finding the LF local Nash equilibrium may be difficult in practice but theoretically important such as second-order analysis.
%Following a local Nash equilibrium concept \cite{???}, it may be necessary to analyze a local LF Nash concept to expand the applicability of MLFMG into real world problems.

\paragraph{Acknowledgement}
%We appreciate the two anonymous reviewers for carefully reading our manuscript and providing insightful comments.
The authors would like to express their appreciation to Prof. Masao Fukushima for his valuable comments.
This work was supported by the Grant-in-Aid for Scientific Research (C) (19K11840) from Japan Society for the Promotion of Science.

% References must be listed in alphabetical order.
%References

\end{document}